\documentclass{amsart}
\usepackage[foot]{amsaddr}
\usepackage{url}
\usepackage{geometry}
\usepackage{float}

\makeatletter
\@namedef{subjclassname@2020}{%
	\textup{2020} Mathematics Subject Classification}
\makeatother

\usepackage[numbers,sort&compress]{natbib}

\usepackage{latexsym}
\usepackage{amsthm,xcolor}
\usepackage{filecontents}
\newtheorem{thm}{Theorem}[section]
\newtheorem{note}[thm]{Note}
\theoremstyle{plain}
\usepackage{bbm}
\newtheorem{lem}[thm]{Lemma}
\newtheorem{prop}[thm]{Proposition}
\usepackage{amsmath}
\usepackage{graphicx}
\usepackage{amssymb}
\usepackage{fullpage}
\usepackage{amsthm}
\newtheorem{asu}[thm]{Assumption}
\newtheorem{defn}[thm]{Definition}
\usepackage{comment}
\usepackage{appendix}
\newtheorem{remark}[thm]{Remark}
\usepackage{amssymb}
\usepackage{footmisc}
\usepackage{verbatim}
\numberwithin{equation}{section}
\usepackage{hyperref}
\usepackage{xcolor}
\usepackage{ulem}
\usepackage{marginnote}
\newcommand{\bb}[1]{\mathbb{#1}}

\theoremstyle{case}

\usepackage{caption}

\let\existstemp\exists
\let\foralltemp\forall
\renewcommand*{\exists}{\existstemp\mkern2mu}
\renewcommand*{\forall}{\foralltemp\mkern2mu}
\usepackage{subcaption}
\DeclareMathOperator*{\plim}{plim}
\definecolor{ecol}{rgb}{0.0, 0.0, 0.0}
\newcommand{\editc}[1]{{\color{black} #1}}

\usepackage[pagewise, displaymath, mathlines]{lineno} 

\usepackage{xpatch} \makeatletter \xpretocmd \start@align{\linenomathWithnumbers}{}{\fail}

\title{Stability of Non-Linear Filter for Deterministic Dynamics}
\subjclass[2000]{Primary 60-01, 62-01; Secondary 05-01.}

\author{Anugu Sumith Reddy$^*$}  \email{anugu.reddy@math.iitb.ac.in}
\author{Amit Apte$^{**}$} \email{apte@icts.res.in}
\address{$^*$Department of Mathematics, Indian Institute of Technology Bombay, Mumbai, India}
\address{$^{**}$International Centre for
  Theoretical Sciences-Tata Institute of Fundamental Research,
  Bangalore, India}

\begin{document}
\normalem
\begin{abstract}
  This papers shows that nonlinear filter in the case of deterministic
  dynamics is stable with respect to the initial conditions under the
  conditions that observations are sufficiently rich, both in the
  context of continuous and discrete time filters. Earlier works on
  the stability of the nonlinear filters are in the context of
  stochastic dynamics and assume conditions like compact state space
  or time independent observation model, whereas we prove filter
  stability for deterministic dynamics with more general assumptions
  on the state space and observation process. We give several examples
  of systems that satisfy these assumptions. We also show that the
  asymptotic structure of the filtering distribution is related to the
  dynamical properties of the signal.
\end{abstract}

\keywords{Data assimilation; Nonlinear filtering; Noise free dynamics; Filter stability}
\subjclass[2020]{Primary: 60G35, 93B07, 93E11, 62M20; Secondary: 93C10 .}

\maketitle

\section{Introduction}\label{sec:intro}
The Bayesian formulation of the data assimilation
problem~\cite{ApteH07, law2015data, reich2015probabilistic,
  asch2016data, carrassi2018data} quite naturally leads to the problem
of non-linear filtering, which had its roots in engineering
applications and for which a rigorous foundational theory had been
established in later half of twentieth
century~\cite{xiong2008introduction, bain2008fundamentals,
  kallianpur1980stochastic}. Filtering aims at estimating state of a
system at a particular instant given some noisy observations of the
system up to that instant. More precisely, we want to study the
evolution of conditional distribution, referred to as \emph{filter} or
\emph{optimal filter} from now on, of the state of the system at time
$t$ given the $\sigma$-algebra generated by the observations made up
to $t$, \editc{where the time $t$ is allowed to be either discrete or continuous}. The
evolution equation of the conditional distribution takes as inputs
the observation path which drives the equation and the initial
condition of the system, \emph{i.e.}, the probability distribution at
the initial time.

In continuous time, this evolution equation is given by
Kushner-Stratanovich (KS) equation whose solution is a measure valued
process (conditional distribution in this case) and initial condition
of KS equation is the probability distribution of initial condition of
the system. In many situations, this initial condition is unknown and
hence, it is desirable to know whether the non-linear filter is
sensitive to the initial condition, at least for large times. In other
words, we desire for solution to KS equation to be asymptotically
stable with respect to the initial condition in case of continuous
time. \editc{ The notion of filter stability is analogously defined in
  the case of discrete time setting.} This property of the filter is
referred to as filter stability~\cite{xiong2008introduction,
  van2007thesis}. Essentially, for stable filters, observations will
correct for the mistake of wrongly initialising the filter, as more
and more observations are made.

In the context of data assimilation in the earth sciences, the signal
or the system being observed is the ocean and/or the atmosphere. Some
of the important characteristics of these systems are that: (i) they
are high dimensional; (ii) the observations are sparse and noisy;
(iii) the dynamical models are very commonly deterministic \cite{palmer2019stochastic},\cite[Section 1.5]{ asch2016data}, and \editc{(iv)}
these systems are chaotic. Thus many numerical algorithms that focus
on one or more of these characteristics are being developed, though
only a few theoretical results related to filtering for
\emph{deterministic, chaotic signal dynamics} have been established so
far. This paper provide filter stability results precisely for such
systems.

\subsection{A summary of previous results}
The problem of filter
stability \editc{has been} studied by many authors under different conditions
on the system and observations. Stability of the filter in case of
Kalman-Bucy filter is studied in \cite{ocone1996asymptotic,
  bishop2017stability} under the conditions of uniform controllability
and uniform observability and in case of Bene\v{s} filters is studied
in \cite{ocone1999asymptotic}. Exponential stability of the filter has
been established in the case of continuous time, ergodic signal and
non-compact domain in \cite{atar1998exponential} and in the case of
discrete time, non-ergodic signal and non-compact domain in
\cite{budhiraja1999exponential}. In \cite{atar1997exponential,
  budhiraja1997exponential}, filter stability is achieved using the
Hilbert projective metric and Birkhoff's contraction inequality. The
filter stability in the case when signal is a general Markov process
with a unique invariant measure under suitable regularity conditions
is studied in \cite{budhiraja2003asymptotic}. In
\cite{clark1999relative} it is proved, using relative entropy
arguments, that some appropriate distance of correctly initialised and
incorrectly initialised conditional distribution of specific functions
of state (namely observation function) goes to zero. Moreover, they
show that the relative entropy of optimal filter with respect to
incorrectly initialised filter is a non-negative supermartingale. We
refer the reader to \cite{chiganskystability, van2007thesis,
  chigansky2009intrinsic} and the references therein, for more details
regarding tools involved and results in the filter stability.

In general, proving filter stability requires ergodicity of the signal
or making sufficiently rich enough observations, the precise form of
the latter condition being observability. Roughly speaking, filtering
model is said to be observable when two observation paths (initialised
with two initial conditions) have same distribution and it implies
that initial conditions are identically distributed. Using this
notion, filter stability is established in \cite{van2010nonlinear} in
discrete time and in \cite{van2009observability} in continuous
time. In \cite{mcdonald2018stability}, the authors used a more general
version of observability to establish filter stability in discrete
time. Note that in all the works mentioned in this section, the signal
is a stochastic dynamical system.

\subsection{Main contributions}
In this paper, we prove in Theorem~\ref{stab} the stability of a
general nonlinear filter with deterministic signal dynamics in
continuous time (and an analogous Theorem~\ref{stabd} in discrete
time).  Previous results of stability with linear deterministic
dynamics and linear observations can be found in
\cite{gurumoorthy2017rank, bocquet2017degenerate} in the case of
discrete time and in \cite{ni2016stability, reddy2019asymptotic} in
the case of continuous time. The problem of accuracy (which is a
measure of deviation of the filter from the signal) of the filter for
deterministic dynamics is studied in \cite{cerou2000long}, and we rely
heavily on the techniques used in that study. In particular, the
stability result in Theorem~\ref{stab} is obtained by first proving,
in Theorem~\ref{th1} (and analogous Theorem~\ref{thd1} in discrete
time), the consistency of the smoother, \emph{i.e.}, the asymptotic
convergence of the conditional distribution of the initial condition
given the observations (which is a particular case of the smoothing
problem).

\subsection{Organization of the paper}
 The notation, the statement of the
problem in continuous time, and the assumptions used are all
introduced in Section~\ref{S3}, along with the significance of the
assumptions in Section~\ref{sig}. We state and prove the main
Theorems~\ref{th1} and~\ref{stab} in Section~\ref{Sresults}. The
same methods as in continuous time are used in discrete time setting
for establishing stability, and we briefly set up the problem discrete
time framework and state the analogous results in
Section~\ref{S4}. Asymptotic behaviour of the support of the
conditional distribution is studied in Section~\ref{S5}. Examples of
systems that satisfy the assumptions are presented in
Section~\ref{S7} and conclusions are given in Section~\ref{S6}.

\section{Continuous time nonlinear smoother and filter}
\label{S3}

\subsection{Setup}\label{setup}

We consider a continuous time dynamical system
$\{\phi_t\}_{t\in \bb{R}}$ on a state space $X$ which is a
$p$-dimensional complete Riemannian manifold with metric $d$ and
volume measure $\sigma$. The initial condition ${x_0}$ follows
a distribution $\mu$ \editc{and has a finite second moment}.  These dynamics are observed partially
through the observation process $Y_t \in \mathbb{R}^n$ in the
following way.
\begin{align}\nonumber
  Y_t&= \int_0^t h(s,\phi_s(x_0))ds+W_t \,,
\end{align}
where, \editc{$h:\mathbb{R}^+ \times X \rightarrow \mathbb{R}^n$ is an
observation function (such that $h(\cdot,\cdot)$ is Borel measurable  and $h(t,\cdot)$ has linear growth) }and $W_t \in \mathbb{R}^n$ the standard Brownian
motion respectively. Moreover, $x_0$ and $W$ are assumed to be
independent. Therefore, the probability space that we consider is
$\big\{X \times C\left( \left[0, \infty\right), \bb{R}^n \right),
  \bb{B}(X) \otimes \bb{B} \left(C \left( \left[0, \infty\right),
      \bb{R}^n \right) \right), \bb{P}=\mu \otimes \bb{P}_W
\big\}$. Here, $\bb{B}(.)$ denotes the Borel $\sigma$-algebra of the
corresponding space and $\bb{P}_W$ is the Wiener measure. Let
$\mathcal{F}^{y}_t = \sigma\{ Y_s : 0 \leq s \leq t\}$ be the
observation process filtration.
      
The main object of interest in the above setup is the filter denoted
by $\pi_t$, that is the conditional distribution of the state
\editc{$x_t = \phi_t(x_0)$} at time $t$ given observations up to that time,
i.e., conditioned on $\mathcal{F}^{y}_t$. We will also study the
smoother denoted by $\pi^0_t$, that is the conditional distribution of
the initial condition $x_0$, again conditioned on $\mathcal{F}^{y}_t$.
It follows from Bayes' rule
(\cite[Theorem~3.22]{xiong2008introduction} and
\cite[Proposition~3.13]{bain2008fundamentals}), that for any bounded
continuous function $g$ on $X$, the smoother is given by

\begin{align}\label{smoothcontdist}
  \pi^0_t(g) := \mathbb{E} \left[g \left(x_0 \right)|
  \mathcal{F}^y_t \right] = \frac{\int_X g \left(x \right) Z \left(t,
  x, Y_{\left[0,t\right]} \right) \mu \left(dx \right)}{\int_{X} Z
  \left(t, x, Y_{\left[0,t\right]} \right) \mu \left( dx \right)} \,,
\end{align}
while the filter is given by
\begin{align}\label{filcontdist}
  {\pi}_t(g) := \mathbb{E} \left[g \left(\phi_t(x_0) \right)|
  \mathcal{F}^y_t \right] = \frac{\int_{X}g \left( \phi_t(x) \right) Z
  \left(t, x, Y_{\left[0,t\right]} \right) \mu(dx)}{\int_{X} Z
  \left(t, x, Y_{\left[0,t\right]} \right) \mu(dx)} \,,
\end{align}
where we use the following definition
\begin{align}\nonumber
  Z\left(t, x, Y_{\left[0,t\right]} \right) := \exp\left( \int^t_0 h
  \left(s, \phi_s \left(x \right) \right)^TdY_s - \frac{1}{2} \int_0^t
  \left\| h \left(s, \phi_s \left(x \right) \right) \right\|^2 ds
  \right) \,.
\end{align}
\editc{Throughout the paper, we use the standard definition
	$\mu(\psi) := \int_{\Omega}\psi d\mu$ for a measure $\mu$ on a
	probability space $(\Omega, \mathcal{B})$ and a measurable function
	$\psi \in \mathcal{L}^1( \Omega, \mathcal{B}, \mu)$, and we use
	Euclidean norm on $\mathbb{R}^m$ for any $m \in \mathbb{N}$ and the
	induced matrix norm. $\plim_{t\to \infty}Q_t$ is the limit in probability of random variables $Q_t$, when the limit exists.}
\subsection{Stability of the filter}\label{stabsec}
If the distribution $\mu$ of the initial condition is unknown, then
choosing an incorrect initial condition with law $\nu$, the
corresponding incorrect filter, denoted by $\bar{\pi}_t$, is given by
\begin{align}\nonumber
  \bar{\pi}_t(g) = \mathbb{E} \left[g \left(\phi_t(x_0) \right)|
  \mathcal{F}^y_t \right] = \frac{\int_{X}g \left( \phi_t(x) \right) Z
  \left(t, x, Y_{\left[0,t\right]} \right) \nu(dx)}{\int_{X} Z
  \left(t, x, Y_{\left[0,t\right]} \right) \nu(dx)} \,.
\end{align}
Then filter is said to be stable if $\pi_t$ and $\bar{\pi}_t$ are
asymptotically close in an appropriate sense. More precisely, in this
paper, we establish the filter stability in the following sense: we
say the filter is stable, if for any bounded continuous
$g:X\rightarrow \mathbb{R}$, we have
\begin{align*}
  \lim_{t\to\infty}\mathbb{E}[\left| \pi_t(g) - \bar{\pi}_t(g)
  \right|]=0 \,,
\end{align*}
for $\nu$ in a class to be specified later (Theorem~\ref{stab}).

One of the two main results of the paper is Theorem~\ref{stab} which
states that optimal filter and incorrect filter merge
\cite{d1988merging} weakly in expectation. In order to achieve this,
we first prove our other main result, which is Theorem~\ref{th1},
establishing the convergence, in an appropriate sense, of the smoother
to the Dirac measure at the initial condition. This result in
Theorem~\ref{th1} is a more general version of the result of
\cite[Proposition~2.1]{cerou2000long}, under more general assumptions,
as explained in Section~\ref{compres}.

\subsection{Main assumptions}
\begin{asu}\label{a2}
  There exists a bounded open set $U \subset X$ with diameter
  $K < \infty$ such that $\overline{\phi_t (U)} \subset U $, for all
  $t > 0$.
\end{asu}

\begin{asu}\label{a1} [Observability]
  There exists $\tau > 0$ such that $\forall t \geq 0$ and
  $x_1,x_2\in U$,
  \begin{align}\label{3}
    \rho_t d(x_1, x_2)^2 \leq \int_t^{t+\tau} \left\| h \left(s,
    \phi_{s-t}(x_1) \right) -h \left(s, \phi_{s-t}(x_2) \right)
    \right\|^2 ds \leq R \rho_t d(x_1,x_2)^2 \,,
  \end{align} 
  where, $\rho_t$ is a positive non-decreasing function such that
  $\lim_{t\to\infty} \rho_t = \infty$ and
  $\lim_{t\to\infty} \frac{\int_0^{t} \rho_{s} ds}{\rho_t} = \infty$,
  $\frac{t\rho_t}{\int_0^{t} \rho_{s} ds} \leq C'< \infty$, for some
  $C'>0$ and $R > 1$.
\end{asu}

\begin{asu}\label{a3}\editc{For $\tau>0$ given by Assumption~\ref{a1} and} 
  $\forall x,\;y\in U$, we have
  $d(\phi_\tau (x), \phi_\tau (y)) \leq C d(x,y)$, for some
  $C = C(\tau) > 1$.
\end{asu}

It follows from  Assumption~\ref{a1} that $\forall \, x, y \in U$,
\begin{align}\label{uplowbound}
  \sum_{i=0}^{\editc{N-1}} \rho_{i\tau}d \left(\phi_{i\tau}(x), \phi_{i\tau}(y)
  \right)^2 \leq \int_0^{t} \left\| h \left(s, \phi_s(x) \right) -
  h \left(s, \phi_s(y) \right) \right\|^2 ds \leq R \sum_{i=0}^{\editc{N}}
  \rho_{i\tau} d \left(\phi_{i\tau}(x), \phi_{i\tau}(y) \right)^2 \,,
\end{align}
where, $N = \lfloor \frac{t}{\tau} \rfloor$. Define,
\begin{align}\label{DN}
  D_N(x,y) := \left(\sum_{i=0}^N \rho_{i\tau} d \left(\phi_{i\tau}(x),
  \phi_{i\tau}(y) \right)^2 \right)^{\frac{1}{2}}
\end{align}
and
\begin{align}\label{dN}
  d_N(x,y) := \underset{0 \leq i \leq {N-1}}{\max} d\left(
  \phi_{i\tau}(x), \phi_{i\tau}(y) \right).
\end{align}
It is straightforward to see
\editc{$\sqrt{\rho_{0}} d_{N+1}(x,y) \leq D_N(x,y) \leq  \sqrt{\rho_{N\tau}(N+1)}
d_{N+1}(x,y)$}, a fact that we use later repeatedly. \editc{We also note that for a fixed $N \geq 0$, $D_N(x,y)$ and
  $d_N(x,y)$ are metrics on $X$ (if we extend the
  Assumption~\ref{a1} to all $x_1, x_2 \in X$). Moreover, the
  metrics $D_N$ and $d_{N}$ are equivalent.}

It also follows from Assumption~\ref{a2} that
$\forall \, x, y \in U$, we have a uniform bound $d_N(x,y) \leq
K$. Indeed, from the invariance of $U$, we have
$\phi_{i\tau} (x), \phi_{i\tau}(y) \in U$ and hence we get
$d(\phi_{i\tau} (x), \phi_{i\tau} (y)) \leq K$ for all $i\geq 0$ \editc{and in particular, for $0\leq i\leq N-1$}.

\begin{asu}\label{a22} 
 \editc{For $\mathcal{V} \subset U \times U$, where  $(U\times U)\backslash \mathcal{V}$ is a $\sigma$-null measure set, and for $(x, y) \in \mathcal{V}$} satisfying $d(x, y) \geq b > 0$, the following holds
  \begin{align*}
    D^2_N(x,y) \geq L^2(b) \sum_{i=0}^N \rho_{i\tau} \,,
  \end{align*}   
  where, $L(b)$ is a positive constant.
\end{asu}

\begin{asu}\label{a4}
	$supp(\mu) \subset U$.
\end{asu}

Before proceeding to the main content of the paper, we define the
notion of the spanning sets \cite[Definition
7.8]{walters1982introduction} which plays an important role in the
proof of Theorem~\ref{th1}. It will help us get the estimates of the
covering number of a compact set with $\epsilon$-balls (under the
metric $d_N$) for any $\epsilon>0$.
\begin{defn}
  For a given compact set $\mathcal{K}\subset X$, $n \geq 0$ and
  $\epsilon > 0$, the set $F \subset X$ is called
  $(n, \epsilon)$-spanning set of $\mathcal{K}$ with respect to
  $\phi_\tau$ if $\forall x \in \mathcal{K}$, $\exists y \in F$ such
  that
  $\underset{0 \leq i \leq n-1}{\editc{\max}} d\left( \phi_{i\tau}(x),
    \phi_{i\tau}(y) \right) \leq \epsilon$.
\end{defn} 
\begin{defn} \label{def:rspan}
  $r_n(\mathcal{K}, \epsilon, \phi_\tau)$ is defined as the minimum
  possible cardinality of $(n, \epsilon)$-spanning sets of
  $\mathcal{K}$.
\end{defn}

Note that for any $n$, $r_n(\mathcal{K}, \epsilon, \phi_\tau)$ is
finite due to compactness of $\mathcal{K}$. The following bound on
this quantity will be used later in proof of Lemma~\ref{auxlem1}.
\begin{lem}\cite[Pg.181]{walters1982introduction} \label{covestimate}
  For a given compact set $\mathcal{K}\subset X$, there exist
  $q=q_\mathcal{K}$ and $b=b_\mathcal{K}$ such that the following
  holds for all $n \geq 0$.
  \begin{align*}
    r_n(\epsilon, \mathcal{K}, \phi_\tau) \leq q (C^nb
    \epsilon^{-1})^p,
  \end{align*}
where, $p$ is the dimension of $X$.
\end{lem}

\subsection{Discussion about the assumptions}\label{sig}
\editc{In the following, we give a qualitative understanding of the assumptions stated in the previous section,} deferring to the Section~\ref{S7} a detailed
discussion of some important examples for which we can explicitly
verify or provide strong numerical evidence for these assumptions.
\begin{enumerate}
\item \textbf{Trapping region:} Assumption~\ref{a2} says that if
  we start inside $U$, then we stay inside $U$
  for all future times. We note that this assumption is not equivalent
  to assuming the state space $X$ to be a compact metric
  space. Indeed, in Section~\ref{examplenoncompact}, we will see
  examples of systems whose state space is non-compact, but which
  satisfy the Assumption~\ref{a2}. We also note that such a
  trapping region $U$ exists for many dynamical systems with chaotic
  behavior on non-compact spaces.

\item \textbf{Initial condition inside the trapping region:} We note
  that Assumption~\ref{a4} is quite natural since it is plausible to
  assume that the state being observed lies inside the trapping region
  $U$, that is to say, a natural system evolving over long enough time
  prior to making observations would have settled in some kind of
  attractor which is in the set $U$. (Also, see
  Remark~\ref{initisupp}.)
  
\item \textbf{Information is lost by the dynamics at most at an
    exponential rate:} The topological entropy (see
  \cite[Pg. 169]{walters1982introduction}) of the dynamics is a
  measure of the rate at which information is lost in a topological
  sense and finite, non-zero topological entropy is interpreted as
  information loss at an exponential rate.  Assumption~\ref{a3}
  implies that the topological entropy is finite \cite[Theorem
  7.15]{walters1982introduction}, thus leading to the at-most
  exponential loss of information.

  To express this more precisely, we consider an open ball, denoted by
  $Q_N(r,x)$, of radius $r$ around $x \in X$ under the metric
  $d_N$. It is clear that for $y, y' \in X$,
  $d_N(y,y') \leq d_{N+1}(y,y'), \; \forall N \geq 0$. Therefore, the
  volume of $Q_N(r,x)$ is non-increasing in $N$. Informally, this
  means that, as $N \to \infty$, the set of all points whose orbits
  are within a $d_N$-distance $r$ from the orbit of $x$ can shrink to
  a zero volume set containing $x$. Assumption~\ref{a3} is used to
  show that this rate is at most exponential (see
  inequality~\ref{voldec}).

\item \textbf{Relation to observability - the information loss by the
    dynamics is compensated by information gained from observations:}
  Assumption~\ref{a1} resembles closely the well-known observability
  condition \cite{anderson1969new},
  \cite[Definition~1]{reddy2019asymptotic} in the linear case except
  for the dependence of $\rho_t$ on $t$ satisfying certain
  conditions. These additional conditions on $\rho_t$ can be
  understood intuitively in the following way. The dynamics loses
  information (as explained above) which can be attributed to
  sensitive dependence of the dynamics on initial conditions. In such
  a case, in order to establish the accuracy of the smoother, we have
  to make observations at a rate faster than the rate at which the
  dynamics loses the information (which is at most exponential).

  To express this more precisely, note that the bound from the
  inequality~\eqref{voldec} mentioned above enters
  inequality~\eqref{ineqexp} (the last term in the exponent).
  Considering $\rho_t$ as mentioned in Assumption~\ref{a1}, i.e. by
  ensuring that we are observing at a fast enough rate, leads to this
  last term going to zero. We also note that the conditions on
  $\rho_t$ stated in Assumption~\ref{a3} may not be optimal.

\item \textbf{Divergence of nearby orbits:} Assumption~\ref{a22}
  says that two orbits, started at a given distance away from each
  other, do not come too close to each other very often. \editc{In particular, this assumption ensures that in
    inequality~\eqref{upper}, the numerator decays to zero at a rate
    higher than the rate at which the denominator goes to zero.} Intuitively, this is reasonable for a system for which $U$ does not
  contain any stable periodic orbits or fixed points, but rather
  contains chaotic attractors. To illustrate this point, we give, in
  Section~\ref{S7}, examples of classes of systems that satisfy this
  assumption. \editc{Our result, which does not apply for systems which contain stable periodic orbits is in contrast to \cite[Theorem 3.3]{Crou1994LongTA} where, under additional assumptions, the filter corresponding to a system with stable periodic orbit is shown to asymptotically concentrate around the true trajectory. This contrast is due to the difference in approaches.}  
  
  We note that for a system that contains stable periodic orbits or
  fixed points, the conclusion of Theorem~\ref{th1} will not hold
  since essentially such a system will ``forget'' the initial
  conditions and the smoother $\pi_t^0$ will not concentrate at the
  initial condition. But on the other hand, for such a ``stable''
  system, the conclusions of Theorem~\ref{stab} about filter
  stability may still be expected to hold, though our approach for
  proving this result using concentration of smoother will not be
  applicable in this case.
\end{enumerate}

\section{Main results}\label{Sresults}

\editc{The two main results of the paper are Theorem~\ref{th1} and~\ref{stab}.}
Theorem~\ref{th1} \editc{states that} the asymptotic in time concentration of the
smoother $\pi^0_t$ to the Dirac measure at the initial condition and Theorem~\ref{stab} states that optimal filter and
incorrect filter merge weakly in expectation.

\begin{thm}\label{th1}
  Suppose $\mu$ is absolutely continuous with respect to the volume
  measure $\sigma$ on $X$ and $\frac{d\mu}{d \sigma}$ is continuous on the
  support of $\mu$. Under the
  Assumptions~\editc{\ref{a2}---
  \ref{a4}},
    for all $a > 0$, there exists
  $\alpha(a)>0$ such that the smoother
  $\pi^0_t:= \mathbb{P} \left[x_0 | \mathcal{F}^y_t \right]$ satisfies
  \begin{align}
   \editc{ \plim_{t\to\infty}e^{\alpha(a) t} \left[ 1 - \pi^0_t \left( B_a(x_0) \right) \right]
    = 0,}
  \label{th1eq} \end{align}
  where $B_a(x_0) := \{x \in X : d(x,x_0) \leq a\}$ is the ball
  centered at $x_0$ and the rate $\alpha(a)>0$ depends only on the
  radius $a$ of the ball.
\end{thm}

\begin{proof}
  In order to show that \editc{$\pi^0_t(B_a(x_0))\xrightarrow[t\to\infty]{P}1$} , we will show,
  in Lemma~\ref{auxlem3}, that \editc{$\pi^0_t(B_a(x_0)^c)\xrightarrow[t\to\infty]{P}0$} at
  an exponential rate.
  
  Recall that for any measurable set $A \in \bb{B}(X)$,
  \begin{align}\nonumber
    \pi^0_t(A)= \frac{\int_{A} \exp \left( \int_0^t h \left(s,
    \phi_s(x) \right)^T dY_s - \frac{1}{2} \int_0^t \left\| h \left(s,
    \phi_s(x) \right) \right\|^2 ds \right) \mu(dx) }{ \int_{X} \exp
    \left( \int_0^t h \left( s, \phi_s(x) \right)^T dY_s - \frac{1}{2}
    \int_0^t \left\| h \left(s, \phi_s(x) \right) \right\|^2 ds \right)
    \mu(dx)}
  \end{align}
  We substitute $ dY_s = h\left(s,\phi_s(x_0)\right)ds+dW_s$ and
  multiply the numerator and the denominator by
  $\exp(\int_0^th(s,\phi_s(x_0))^TdW_s-\frac{1}{2}\int_0^t
  \|h(s,\phi_s(x_0))\|^2 ds )$, which is independent of $x$ to get,
  \begin{align}
    \pi^0_t(A)= \frac{\int_{A}\exp\left(\int_0^t
    \left[h\left(s,\phi_s(x)\right) - h \left(s, \phi_s(x_0) \right)
    \right]^TdW_s - \frac{1}{2}\int_0^t \left\| h \left(s, \phi_s(x)
    \right) - h \left(s, \phi_s(x_0) \right) \right\|^2 ds \right)
    \mu(dx)}{\int_{X} \exp \left( \int_0^t
    \left[h\left(s, \phi_s(x) \right) -h \left(s, \phi_s(x_0) \right)
    \right]^T dW_s - \frac{1}{2} \int_0^t
    \left\|h \left(s, \phi_s(x) \right) -h \left(s, \phi_s(x_0) \right)
    \right\|^2 ds \right)\mu(dx)} \,.
  \end{align}
  Define
  $A_s (x,x_0):= \left[ h \left(s, \phi_s(x) \right) -h \left(s,
      \phi_s(x_0) \right) \right]$, the set
  $Q_N(r,x):= \{y \in X: d_N(x,y) < r \}$ for $r > 0$, and
  $N := \lfloor \frac{t}{\tau} \rfloor$

  We now consider,
  \begin{align}\nonumber
    \pi^0_t(B_a(x_0)^c)
    &= \frac{\int_{B_a(x_0)^c}\exp\left(\int_0^t
      A_s(x,x_0)^TdW_s-\frac{1}{2}\int_0^t
      \left\|A_s(x,x_0)\right\|^2 ds
      \right)\mu(dx)}{\int_{X}\exp\left(\int_0^t
      A_s(x,x_0)^TdW_s-\frac{1}{2}\int_0^t
      \left\|A_s(x,x_0)\right\|^2 ds
      \right)\mu(dx)}\\\nonumber
    &\leq \frac{ \int_{B_a(x_0)^c}\exp\left(\int_0^t
      A_s(x,x_0)^TdW_s-\frac{1}{2}  \sum_{i=0}^{\editc{N-1}}
      \rho_{i\tau}d\left(\phi_{i\tau}(x), \phi_{i\tau}(x_0) \right)^2
      \right) \mu(dx)}{\int_{Q_N(r,x_0)} \exp\left(\int_0^t
      A_s(x,x_0)^T dW_s - \frac{R}{2} \sum_{i=0}^{\editc{N}}
      \rho_{i\tau}d\left(\phi_{i\tau}(x),\phi_{i\tau}(x_0)\right)^2 
      \right)\mu(dx)}
      \quad \textrm{using \eqref{uplowbound}} \\ \label{upper}
    &\leq \frac{ \int_{B_a(x_0)^c}\exp\left(-\sum_{i=0}^{\editc{N-1}}
      \rho_{i\tau}d\left(\phi_{i\tau}(x), \phi_{i\tau}(x_0) \right)^2
      \left(-\smash{\displaystyle\sup_{x\in
      B_a(x_0)^c}}\frac{\left| \int_0^t A_s(x,x_0)^T dW_s
      \right|}{\sum_{i=0}^{\editc{N-1}} \rho_{i\tau}
      d\left(\phi_{i\tau}(x),\phi_{i\tau}(x_0)\right)^2}+\frac{1}{2}
      \right) \right)\mu(dx)}{\int_{Q_N(r,x_0)}\exp\left(\int_0^t
      A_s(x,x_0)^TdW_s-\frac{R}{2} \sum_{i=0}^{\editc{N}}
      \rho_{i\tau}d\left(\phi_{i\tau}(x),\phi_{i\tau}(x_0)\right)^2 
      \right) \mu(dx)},
  \end{align}  
  where we used the fact that 
  \begin{align*}
    \frac{\left|\int_0^t A_s(x,x_0)^TdW_s\right|}{\sum_{i=0}^{\editc{N-1}}
    \rho_{i\tau}
    d\left(\phi_{i\tau}(x),\phi_{i\tau}(x_0)\right)^2}\leq \sup_{x\in
    B_a(x_0)^c}\frac{\left| \int_0^t A_s(x,x_0)^T dW_s
    \right|}{\sum_{i=0}^{\editc{N-1}} \rho_{i\tau}
    d\left(\phi_{i\tau}(x),\phi_{i\tau}(x_0)\right)^2}
  \end{align*}  
  From \eqref{upper}, it is clear that in order to establish our
  desired result, it is sufficient to find suitable estimates on
  \editc{\begin{align}\label{eq:sup1}
\Delta_a(x_0,t):=    \sup_{x\in B_a(x_0)^c} \frac{\left|\int_0^t
    A_s(x,x_0)^TdW_s\right|}{\sum_{i=0}^{\editc{N-1}}
    \rho_{i\tau}d\left(\phi_{i\tau}(x),\phi_{i\tau}(x_0)\right)^2}
    \qquad \textrm{and} \qquad    \sup_{x\in Q_N(r,x_0)}\left|\int_0^t
    A_s(x,x_0)^TdW_s\right| \,.
  \end{align}
It will be useful to define the following quantity:
\begin{align}\label{eq:sup2}
\Gamma_a(x_0,t):=\sup_{x\in B_a(x_0)}\left|\int_0^t
A_s(x,x_0)^TdW_s\right|
\end{align}}
Since $t$ uniquely determines $N$, we have omitted the dependency on $N$ in ~\eqref{eq:sup1}. 
The relevant bounds on \editc{$\Delta_a(x_0,t)$ and $\Gamma_a(x_0,t)$} are stated in Lemmas~\ref{auxlem1}--\ref{auxlem2}.
  
  \begin{lem}\label{auxlem1}
    $\forall a>0$ and $\forall t \geq \tau$ with
    $N = \lfloor \frac{t}{\tau} \rfloor$, we have
    \begin{align} \label{eq:boundlem1}
      \mathbb{E} \left[\editc{\Gamma_a(x_0,t)}\right]
      & \leq \editc{S(N,a)}  \,,
    \end{align}
    with $b_{\mathcal{K}(a)}, q_{\mathcal{K}(a)}$ being the constants from Lemma~\ref{covestimate}, while
    $C = C(\tau), K, R$ are from Assumptions~\ref{a1}--\ref{a3} \editc{and 
    \begin{align}\label{eq:SNA}
    S(N,a):=
    24 \sqrt{(N+1)R\rho_{N\tau}}\left(K\left(p\left(N+1\right)\log\left (C\right)+\log
    	\left(q_{\mathcal{K}(a)}b_{\mathcal{K}(a)}^p\right)\right)^{\frac{1}{2}}+ 2 \sqrt{Kp}\right)
    \end{align}}
  \end{lem}
  \begin{proof}
    Since $x_0$ and \editc{$W_{(\cdot)}$} are independent,
    \begin{align} \label{eq:ebound1}
      \mathbb{E}\left[\editc{\Gamma_a(x_0,t)}\right] \leq
      2 \editc{\mathbb{E}\left[ \mathbb{E} \left[ \sup_{x\in B_a(x_0)}\int_0^t
      A_s(x,x_0)^TdW_s\Big|\mathcal{F}^{x_0}\right]\right]}, 
    \end{align}
    where, $\mathcal{F}^{x_0}$ is the $\sigma$-algebra generated by $x_0$. Observing that $\int_0^t A_s(x,x_0)^TdW_s$
    is a centered \editc{Gaussian} process, we use the result
    \cite[Theorem 6.1]{ledoux1996isoperimetry}
    \begin{align} \label{eq:ebound2}
      \mathbb{E} \left[\sup_{B_a(x_0)} \int_0^t A_s(x,x_0)^TdW_s
      \Big|\mathcal{F}^{x_0}\right] \leq 24 \int_0^\infty \log^{\frac{1}{2}} \left(N
      \left(B_a(x_0), \bar{d}_t, \epsilon \right) \right) d\epsilon \,,
    \end{align}
    where, $N \left(B_a(x_0), \bar{d}_t, \epsilon \right)$ is the
    minimum number of balls of radius $\epsilon$ under the
    psuedo-metric $\bar{d}_t$ required to cover $B_a(x_0)$ (which is finite for all $\epsilon$ due to the
      compactness of $B_a(x_0)$), where,    
    \begin{align}\nonumber
      \bar{d}_t(x,y)
      &:= \sqrt{\mathbb{E}_W \left[ \left( \int_0^t A_s(x,x_0)^TdW_s -
        \int_0^t A_s(y,x_0)^T dW_s \right)^2 \right]}\\\nonumber
      &= \sqrt{ \int_0^t \left\| h \left(s, \phi_s(x) \right) - h
        \left(s, \phi_s(y) \right) \right\|^2 ds} \,.
    \end{align}    
    From \eqref{uplowbound}, \editc{it} is clear that,
    \begin{align}\nonumber
      \bar{d}_{t}(x,y) \leq \sqrt{R} D_{\editc{N}}(x,y) \leq \editc{\sqrt{(N+1) R
      \rho_{N\tau}} d_{N+1}(x,y)} \,,
    \end{align}
    which implies that 
    \begin{align}\nonumber
      N\left(B_a(x_0), \bar{d}_{t}, \epsilon \right) \leq N\left(
      B_a(x_0), \sqrt{R} D_{\editc{N}}, \epsilon \right) \leq N\left(
      B_a(x_0), \editc{\sqrt{(N+1) R \rho_{N\tau}} d_{N+1}}, \epsilon \right)
      \,
    \end{align}
 \editc{
	Denoting
	$\bar{\epsilon}(a,N):= \sqrt{(N+1)\rho_{N\tau}R} \sup_{x, y \in
		B_a \left( x_0 \right)} d_{N+1} \left(x, y \right)$, we get the
	following bound:
	\begin{align}\nonumber
	\int_0^\infty \log^{\frac{1}{2}} \left( N \left(B_a(x_0),
	\bar{d}_{t}, \epsilon \right) \right) d\epsilon 
	&\leq \int_0^{\bar{\epsilon}(a,N)} \log^{\frac{1}{2}} \left(N
	\left(B_a(x_0), \sqrt{(N+1) R \rho_{N\tau}} d_{N+1}, 
	\epsilon \right) \right) d\epsilon \\\nonumber
	&= \int_0^{\bar{\epsilon}(a,N)} \log^{\frac{1}{2}} \left(N
	\left(B_a(x_0), d_{N+1}, \epsilon \left( \sqrt{(N+1) R
		\rho_{N\tau}} \right)^{-1} \right) \right) d\epsilon
	\\\label{entr}
	&= \sqrt{(N+1) R \rho_{N\tau}}
	\int_0^{\frac{\bar{\epsilon}(a,N)}{\sqrt{(N+1)R\rho_{N\tau}}}}
	\log^{\frac{1}{2}} \left(N \left(B_a(x_0), d_{N+1}, \beta
	\right) \right) d\beta \,.
	\end{align}
In the following, we compute the upper bound of the integral in the last inequality. To that end, define
\begin{align*}
\delta(a,N):= \frac{\bar{\epsilon}(a,N)}{\sqrt{(N+1)\rho_{N\tau}R}}
\end{align*}
and from the definition of $\bar{\epsilon}(a,N)$, we have $\delta(a,N)\leq K$

\begin{align*}\nonumber
&
\int_0^
{\delta(a,N)}
\log^{\frac{1}{2}}\left(N\left(B_a(x_0),d_{N+1},\beta\right)\right)d\beta\\
&= \int_0^
{\delta(a,N)}
\log^{\frac{1}{2}}\left(r_{N+1}\left( \beta, B_a(x_0),
\phi_\tau \right) \right) d\beta, \text{ from the definitions of $N(\cdot,\cdot,\cdot)$ and $r_{(\cdot)}(\cdot, \cdot,\cdot)$}\\     
&\leq 
\int_0^{\delta(a,N)}
\log^{\frac{1}{2}} \left( q_{\mathcal{K}(a)}
\left(C^{N+1}b_{\mathcal{K}(a)}\beta^{-1}\right)^p\right)d\beta, \text { from Lemma~\ref{covestimate} with $\mathcal{K}=\mathcal{K}(a):= \overline{ \cup_{y\in  U}     \{x: d(x,y)\leq a\} }$ }\\\nonumber
&\leq 
\int_0^{\delta(a,N)}
\left( \log \left (q_{\mathcal{K}(a)} (C^{N+1}b_{\mathcal{K}(a)})^p \right) + p \log \left(
\beta^{-1} + 1  \right) \right)^\frac{1}{2} d\beta
\end{align*}
Using the following inequalities: for $x,y>0$
\begin{align*}
\left(\log \left(1+\frac{1}{x}\right)\right)^{\frac{1}{2}}\leq \frac{1}{\sqrt{x}} \text{ and } \sqrt{x+y}\leq \sqrt{x}+\sqrt{y}
\end{align*}
and computing the resulting integral, we have
\begin{align}\nonumber 
\int_0^
{\delta(a,N)}&
\log^{\frac{1}{2}}\left(N\left(B_a(x_0),d_{N+1},\beta\right)\right)d\beta\\\nonumber 
&\leq \delta(a,N) \left( p \left(N+1 \right) \log C + \log \left(q_{\mathcal{K}(a)}b_{\mathcal{K}(a)}^p \right)\right)^{\frac{1}{2}} + 2 \sqrt{\delta(a,N)
	p}\\\label{eq:lem32finalineq} 
&\leq K \left( p \left(N+1 \right) \log C + \log \left(q_{\mathcal{K}(a)}b_{\mathcal{K}(a)}^p \right)\right)^{\frac{1}{2}} + 2 \sqrt{K
	p}, \text{ from the bound on $\delta(a,N)$.}
\,.
\end{align}

Combining the inequalities~\eqref{eq:ebound1},
	~\eqref{eq:ebound2}, and~\eqref{entr} with~\eqref{eq:lem32finalineq}
	gives ~\eqref{eq:boundlem1}, completing the proof of the lemma.}
  \end{proof}
  As noted earlier, we also need to have estimate on
  \editc{$\Delta_a(x_0,t)$}
  which is given by the lemma below.
  
  \begin{lem}\label{auxlem2}
    $\forall a>0$, $\forall t \geq \tau$ and
    $N=\lfloor\frac{t}{\tau}\rfloor$, 
    \begin{align}\nonumber
      \mathbb{E} \left[ \editc{\Delta_a(x_0,t)} \right] \leq
      \frac{S(N,a)}{\editc{L^2(a)}\sum_{i=0}^{\editc{N-1}} \rho_{i\tau}} \,,
    \end{align}
    where, \editc{$S(N,a)$ is as defined in Equation~\eqref{eq:SNA} and $L(a)$ is a constant from Assumption~\ref{a22}}.
  \end{lem}
  \begin{proof}  
\editc{From Assumption~\ref{a4}, without loss in generality, fix $a>0$ such that $a\leq K$, where $K$ is the diameter of $U$ (in Assumption~\ref{a2}). It then follows from Assumption~\ref{a4} that  $\forall x \in supp(\mu)\cap B_a(x_0)^c$. In other words, we have
    \begin{align}\label{eq:subsets}
supp(\mu)\cap B_a(x_0)^c\subset supp(\mu)\cap \overline{B_a(x_0)^c}\subset \overline{B_a(x_0)^c}\cap B_{K}(x_0)
    \end{align}
Using this notation, we obtain the required bound as follows:
\begin{align}\nonumber
\mathbb{E} \left[ \editc{\Delta_a(x_0,t)} \right]&\leq\mathbb{E}\left[\sup_{x\in \overline{B_a(x_0)^c}\cap B_K(x_0)}
\frac{ \left| \int_0^t A_s(x,x_0)^T dW_s \right|}{\sum_{i=0}^N
	\rho_{i\tau} d \left(\phi_{i\tau}(x), \phi_{i\tau}(x_0)
	\right)^2}\right], \text{ from~\eqref{eq:subsets}}\\
&\leq 
 \frac{1}{L^2(\editc{a}) \sum_{i=0}^{\editc{N-1}} \rho_{i\tau}}
 \mathbb{E}
\left[ \editc{\Gamma_{\editc{K}}(x_0,t)}  \right],\editc{\text{ from Assumption~\ref{a22} and definition of $\Gamma_K(x_0,t)$}}  \\\nonumber
&\leq 
  \frac{1}{L^2(a)\sum_{i=0}^{\editc{N-1}} \rho_{i\tau}}
S(N,\editc{K}), \text{ from Lemma~\ref{auxlem1}}
\end{align}}
 This completes the proof of the lemma.
  \end{proof}
\editc{In the following, we will also need the limit of $\mathbb{E} \left[ \editc{\Delta_a(x_0,t)}
	\right]$.
	\begin{lem} $\forall a > 0$,
		\begin{align}\label{exp0}
		\lim_{t\to\infty} \mathbb{E} \left[ \editc{\Delta_a(x_0,t)}
		\right] = 0 \,.
		\end{align}
	\end{lem}
	  \begin{proof} 
We note that $t\to \infty \Leftrightarrow N\to\infty$ and
  $\frac{N\rho_{N\tau}}{\sum_{i=0}^N \rho_{i\tau}}\leq C'$ (from Assumption~\ref{a1}). And also, from the definition of $S(N,a)$ and from Lemma~\ref{auxlem2}, we know that, for large $N$,
  \begin{align*}
  \mathbb{E} \left[ {\Delta_a(x_0,t)}\right]\leq O\left(\frac{N\sqrt{\rho_{N\tau}}}{\sum_{i=0}^{\editc{N-1}} \rho_{i\tau}}\right)= O\left(\frac{N\rho_{N\tau}}{\left(\sum_{i=0}^{\editc{N-1}} \rho_{i\tau}\right)\sqrt{\rho_{N\tau}}}\right)
  \end{align*} 
  Now using the fact that $\rho_t\uparrow \infty$, it suffices to show that 
  \begin{align*}
\liminf_{N\to \infty}  \frac{\sum_{i=0}^{\editc{N-1}} \rho_{i\tau} }{N\rho_{N\tau}} >C'', \text{ for some $C''>0$}
  \end{align*}
  To show this, consider
  \begin{align*}
  \frac{\sum_{i=0}^N \rho_{i\tau}}{N\rho_{N\tau}}=\frac{\sum_{i=0}^{N-1}\rho_{i\tau}}{N\rho_{N\tau}} +\frac{1}{N}\geq  (C')^{-1}>0\\   
  \end{align*}
  Taking limit inferior as $N\to \infty$ in the above inequality will give us the desired bound and proves the result.
\end{proof}} 
    
  Finally, we need the lemma below to complete the proof of
  Theorem~\ref{th1}.
  \begin{lem}\label{auxlem3}
    $\forall a>0$, $\exists\alpha=\alpha(a) >0$ such that
    $\plim_{t\to\infty} e^{\alpha
      t}{\pi^0_{t}\left(B_a(x_0)^c\right)}=0.$
  \end{lem}
  \begin{proof}
    From \eqref{exp0}, we have 
    \begin{align*}
      \plim_{t\to\infty}\editc{\Delta_a(x_0,t)}=0.
    \end{align*}
    Recall that $t\to \infty \Leftrightarrow N\to \infty$. In
    particular, the above equation holds for any subsequence $\{t_j\}$
    . Therefore, there is sub-subsequence $\{t_{j_q}\}$ such that
    \begin{align*}
      \lim_{q\to \infty}\editc{\Delta_a(x_0,t_{j_q})}=0,\text{
      $\bb{P}$-a.s.}
    \end{align*}

    From the above, for large enough $q$, we have 
    \begin{align*}
   \editc{ \Delta_a(x_0,t_{j_q})}<  \frac{1}{4}
    \end{align*}
    and thereby,
    \begin{align}\nonumber
      \int_{B_a(x_0)^c}
      &\exp\left(-\sum_{i=0}^{\editc{N(j,q)-1}}\rho_{i\tau}
        d\left(\phi_{i\tau}(x),\phi_{i\tau}(x_0)\right)^2 \left(\editc{-\Delta_a(x_0,t_{j_q})}+\frac{1}{2}
        \right) \right)\mu(dx)\\\nonumber
      &\leq \int_{B_a(x_0)^c}\exp\left(-\sum_{i=0}^{\editc{N(j,q)-1}}
        \rho_{i\tau}d\left(\phi_{i\tau}(x),\phi_{i\tau}(x_0)\right)^2
        \frac{1}{4} \right)\mu(dx)\\\label{upper1}
      &\leq \exp\left(-\frac{1}{4}L^2(a)\sum_{i=0}^{\editc{N(j,q)-1}}\rho_{i\tau}\right),
    \end{align}
    where, $N(j,q):=\lfloor\frac{t_{j_q}}{\tau}\rfloor$ and we used \editc{Assumption~\ref{a22}} together with the fact that
    $\mu\left(B_a(x_0)^c\right)\leq 1$. We now consider
    \begin{align}\nonumber
      \int_{Q_N(r,x_0)}
      &\exp\left(\int_0^t
        A_s(x,x_0)^TdW_s-\frac{1}{2}\int_0^t \left|A_s(x,x_0)\right|^2 ds
        \right)\mu(dx)\\\nonumber
      &\geq \int_{Q_N(r,x_0)}\exp\left(\int_0^t
        A_s(x,x_0)^TdW_s-\frac{1}{2} R \sum_{i=0}^{\editc{N}}\rho_{i\tau}
        d\left(\phi_{i\tau}(x),\phi_{i\tau}(x_0)\right)^2\right)\mu(dx)
      \\\nonumber
      &\geq \int_{Q_N(r,x_0)}\exp\left(\int_0^t
        A_s(x,x_0)^TdW_s-\frac{1}{2}R d^2_{N+1}(x,x_0)\sum_{i=0}^{\editc{N}}
        \rho_{i\tau}\right)\mu(dx)\\\label{lower1}
      &\geq \int_{Q_N(r,x_0)}\exp\left(-\sum_{i=0}^{\editc{N}}
        \rho_{i\tau}\left(-\frac{\int_0^t A_s(x,x_0)^TdW_s}{\sum_{i=0}^{\editc{N}}
        \rho_{i\tau}}+\editc{\frac{R}{2} r^2}\right)\right)\mu(dx),
    \end{align}
    In the last inequality, we used the definition of
    $Q_N(r,x_0)$. And also, from the definition of $Q_N(r,\editc{x_0})$, it is
    clear that  $Q_N(r,\editc{x_0})\subset B_r(\editc{x_0})\subset B_a(x_0)$, \editc{for $r\leq  a$}. Therefore,
    \begin{align*}
      \mathbb{E}\left[\sup_{Q_N(r,x_0)}\left|\int_0^t
      A_s(x,x_0)^TdW_s\right|\right]\leq
      \mathbb{E}\left[\editc{\Gamma_r(x_0,t)}\right]\leq
      \mathbb{E}\left[\editc{\Gamma_a(x_0,t)}\right]
    \end{align*}
    From Lemma~\ref{auxlem1}, it follows that 
    \begin{align*}
      \frac{1}{\sum_{i=0}^{\editc{N}}
      \rho_{i\tau}}\mathbb{E}\left[\sup_{Q_N(r,x_0)}\left|\int_0^t
      A_s(x,x_0)^TdW_s\right|\right]\leq
      \frac{\editc{S(N,a)}}{\sum_{i=0}^{\editc{N}} \rho_{i\tau}}
    \end{align*} 
\editc{From computations similar to those used in showing Equation~\eqref{exp0}, the right hand side of above inequality}
    converges to zero as $t\to \infty$ which again implies that
    \begin{align*}
      \plim_{t\to\infty}\frac{\editc{\Gamma_r(x_0,t)}}{\sum_{i=0}^{\editc{N}} \rho_{i\tau}}=0.
    \end{align*}
    In particular, it converges to zero in probability on subsequence
    $t_j$. Therefore, we can choose a sub-subsequence, $\{t_{j_q}\}$
    (that works for the previous scenario) such that
    \begin{align*}
      \lim_{q\to\infty}\frac{\sup_{Q_{N(j,q)}(r,x_0)}\left|\int_0^{t_{j_q}}
      A_s(x,x_0)^TdW_s\right|}{\sum_{i=0}^{\editc{N(j,q)}} \rho_{i\tau}}=
      0,\;\; \text{$\bb{P}$-a.s}.
    \end{align*}
    For large enough $q$,
    \begin{align*}
      \frac{\sup_{Q_{N(j,q)}(r,x_0)}|\int_0^{t_{j_q}}
      A_s(x,x_0)^TdW_s|}{\sum_{i=0}^{\editc{N(j,q)}}
      \rho_{i\tau}}<\editc{\frac{Rr^2}{2}}
    \end{align*} 
    Therefore, \eqref{lower1} becomes
    \begin{align}\nonumber
      \int_{Q_{N(j,q)}(r,x_0)}\exp\left(-\sum_{i=0}^{\editc{N(j,q)}}
      \rho_{i\tau}\left(-\frac{\int_0^t
      A_s(x,x_0)^TdW_s}{\sum_{i=0}^{\editc{N(j,q)}}
      \rho_{i\tau}}+\editc{\frac{Rr^2}{2}} \right)\right)\mu(dx)
      &\geq \int_{Q_{N(j,q)}(r,x_0)}\exp\left(-\sum_{i=0}^{\editc{N(j,q)}}
        \rho_{i\tau}\editc{Rr^2} \right)    \mu(dx) \\\label{lower2}
      &\geq \exp\left(-\sum_{i=0}^{\editc{N(j,q)}}
        \rho_{i\tau}\editc{Rr^2}\right)\mu(Q_{N(j,q)}(r,x_0))
    \end{align}
    Combining inequalities \eqref{lower2} and \eqref{upper1}, we have
    \begin{align}\label{finup}
      \pi^0_{t_{j_q}}(B_a(x_0)^c)\leq
      \frac{\exp\left(-\sum_{i=0}^{\editc{N(j,q)-1}}
      \rho_{i\tau}\left(\frac{L^2(a)}{4} - \editc{Rr^2}\right) +\rho_{\editc{N(j,q)\tau}}\editc{Rr^2}\right)}{\mu\left(Q_{N(j,q)}(r,x_0)\right)}
    \end{align}
    As mentioned in Section~\ref{sig}, in general, the set
    $Q_n(r,x_0)$ will shrink to a set containing $x_0$ (which is not
    open) as $n\to \infty$. \editc{This is because for chaotic systems, $\phi_t(x)$ depends very sensitively on $x$ after large times.}. We will see that
    $\mu\left(Q_{N(j,q)}(r,x_0)\right)$ goes to zero at most at an
    exponential rate.

    From the assumption of absolute continuity of $\mu$ with respect
    to $\sigma$ , we have $\frac{d\mu}{d\sigma}(x_0)>0$ $\bb{P}-\text{a.s.}$
    From the continuity of $\frac{d\mu}{d\sigma}$, there exist $r_1 >0$
    and $C_1>0$ such that $\frac{d\mu}{d\sigma} (x)>C_1$, for any
    $x\in B_{r_1}(x_0)$. Therefore, with the help of Radon-Nikodym
    Theorem and choosing $r< r_1$, we have
    \begin{align}\label{vollow}
      \mu\left(Q_{N(j,q)}(r,x_0)\right)> C_1
      \sigma\left(Q_{N(j,q)}(r,x_0)\right).
    \end{align}
    From the Assumption~\ref{a3}, we have the following:
    \begin{align*}
      d_N(x,y)&\leq C^Nd(x,y)\\
      \implies B_{\frac{r}{C^N}}(x_0)&\subset Q_N(r,x_0)
    \end{align*}
    \eqref{vollow} becomes 
    \begin{align} \nonumber
      \mu\left(Q_{N(j,q)}(r,x_0)\right)> C_1
      \sigma\left(Q_{N(j,q)}(r,x_0)\right)
      &> C_1 \sigma\left(B_{\frac{r}{C^{N(j,q)}}}(x_0)\right)\\\label{voldec}
      &> C_1C_2 \left(\frac{r}{C^{N(j,q)}}\right)^p,
    \end{align}
    for some $C_2=C_2(p,\mathcal{K})$ (with $\mathcal{K}= \overline{ \cup_{y\in  U}     \{x: d(x,y)\leq a\}}$) and \eqref{finup} becomes
    \begin{align*}
      \pi^0_{t_{j_q}}(B_a(x_0)^c)\leq
      \frac{\exp\left(-\sum_{i=0}^{N(j,q)}
      \rho_{i\tau}\left(\frac{L^2(a)}{4} - \editc{Rr^2}\right) +\rho_{\tau
      \left(N(j,q)+1\right)}\editc{Rr^2}\right)}{C_1C_2
      \left(\frac{r}{C^{N(j,q)}}\right)^p}
    \end{align*}
    \begin{align}\label{ineqexp}
      \pi^0_{t_{j_q}}(B_a(x_0)^c)\leq \frac{1}{C_1C_2
      r^p}\exp\left(-\sum_{i=0}^{\editc{N(j,q)-1}}\rho_{i\tau}
      \left(\left(\frac{L^2(a)}{4} - \editc{Rr^2}\right) -\frac{\rho_{
      \editc{N(j,q)\tau}}\editc{Rr^2}}{\sum_{i=0}^{\editc{N(j,q)-1}}\rho_{i\tau}}-
      \frac{ N(j,q)\log_e C^p
      }{\sum_{i=0}^{\editc{N(j,q)-1}}\rho_{i\tau}}\right)\right)
    \end{align}
    Choosing $r$ small enough such that $\frac{L^2(a)}{4} - \editc{Rr^2} >0$ and
    from Assumption~\ref{a1} ($\lim_{t\to\infty} \rho_t = \infty$
    and
    $\lim_{t\to\infty} \frac{\int_0^{t} \rho_{s} ds}{\rho_t} =
    \infty$), for large enough $q$, the sum in the exponent can be
    made positive which results in
    $\pi^0_{t_{j_q}}\left(B_a(x_0)^c\right)$ converging exponentially
    to zero almost surely as $q\to\infty$. Since, the subsequence
    $t_j$ is arbitrary, it implies that
    $\pi^0_{t}\left(B_a(x_0)^c\right)$ converges exponentially to zero
    \editc{in probability} as $t\to\infty$.
  \end{proof}
  From Lemma~\ref{auxlem3}, it is clear that the assertion of the
  Theorem~\eqref{th1} follows.
\end{proof}

In the previous theorem, we established that conditional distribution
of $x_0$ given observations is asymptotically supported only on balls around $x_0$ of arbitrary small radius. In the following, we extend
the previous statement to any measurable set, $A\in \bb{B}(X)$.

\begin{prop}\label{pr1}
  Under the hypotheses of Theorem~\ref{th1},
  $\editc{\plim_{t\to\infty}\pi^0_t (A)=} 0,\;\; \forall A
  \in\bb{B}(X),\;\;x_0\notin A $
\end{prop}
\begin{proof}
\editc{If we consider an arbitrary sequence $t_j \to \infty$, there exists a subsequence that is still denoted by  $t_{j}$ such that 
	\begin{align*}
	\lim_{j\to\infty} e^{\alpha(a)t_j} \left( \pi^0_{t_j}(B_a(x_0))-1\right)=0,\text{ $\bb{P}$-a.s.} 
	\end{align*}
  It can be seen easily that the conclusion of the theorem holds even
  if $d(x,x_0)\leq a$ is replaced with $d(x,x_0)< a$. Indeed, fixing $\rho <a$ and $\gamma>0$, we have 
  \begin{align}\nonumber
    e^{\gamma t_j}\left(\pi^0_{t_j} \left(B_{a-\rho} (x_0)\right)-1\right)\leq e^{\gamma t_j}\left(\pi^0_{t_j} \left(\left\{x \in X : d(x,x_0)<
    a\right\}\right)-1\right)
     \leq e^{\gamma t_j}\left(\pi^0_{t_j} \left(B_a (x_0)\right)-1\right),
  \end{align}
Now choosing $\gamma= \min\{\alpha(a),\alpha(a-\rho)\}$, we can conclude that 
\begin{align*}
e^{\gamma t_j}(\pi^0_{t_j} (B_{a-\rho} (x_0))-1)\xrightarrow{ }\;0 \text{ and } e^{\gamma t_j}(\pi^0_{t_j} (B_{a} (x_0))-1)\xrightarrow{}\;0, \text{ as } j\to\infty, \text{ $\bb{P}$-a.s.}
\end{align*}
  and thereby,  
  \begin{align}\label{eq:openball}
\lim_{j\to\infty}e^{\gamma t_j}\left(\pi^0_{t_j} \left(\left\{x \in X : d(x,x_0)< a\right\}\right)-1\right)=0 , \;\; \forall a>0,\text{ $\bb{P}$-a.s.}
\end{align}
Since for any open set $\mathcal{O}$ containing $x_0$, there exists $r>0$ such that $\left\{x \in X : d(x,x_0)<r\right\}\subset \mathcal{O}$, we have
  \begin{align}\label{open}
\lim_{j\to\infty}  \pi^0_{t_j} \left(\mathcal{O}^c\right)&\leq \lim_{j\to\infty} \pi^0_{t_j} \left(\left\{x \in X : d(x,x_0)\geq r\right\}\right)
= 0, \text{ $\bb{P}$-a.s.},\text{ from~\eqref{eq:openball}}.   
\end{align}
For an open set $\mathcal{O}_1$ not containing $x_0$, there exists $r>0$ such that $\left\{x \in X : d(x,x_0)<r\right\}\cap U_1=\emptyset$. From this and~\eqref{eq:openball}, it is trivial to see that 
\begin{align*}
\lim_{j\to\infty} \pi^0_{t_j}(\mathcal{O}_1)=0, \text{ $\bb{P}$-a.s.}
\end{align*}
And also, for any closed set $\mathcal{C}$, applying the above argument for $\mathcal{C}^c$, we obtain
  \begin{align}\label{close}
  \lim_{j\to\infty}\pi^0_{t_j} (\mathcal{C}) =\begin{cases} 
    &    1, \text{ $\bb{P}$-a.s.}, \text{ if $x_0 \in \mathcal{C}$}\;  \;\;  \;\\
                                & 0, \text{ $\bb{P}$-a.s.},  \text{ if $x_0 \notin \mathcal{C}$.}   
  \end{cases} 
  \end{align}
  Finally, to extend it to all measurable sets, we use the property of
  regular probability measure with Borel $\sigma$-algebra of a metric
  space \cite[Theorem 1.1]{billingsley1999convergence}.
   
  By \cite[Theorem 1.1]{billingsley1999convergence}, for every
  measurable set $A\in \bb{B}(X)$, there exist closed set $C_0$, open
  set $U_0$ such that $\mathcal{C}\subset A \subset \mathcal{O}$ and
  $\pi^0_{t_j}(U_0/C_0)<\frac{1}{2}$.
   
  Let $A$ be such that $x_0 \in A$ which implies that $x_0 \in
  \mathcal{O}$. Choose $0<\eta<\frac{1}{4}$ and $j$ large enough such that
  $\pi^0_{t_j}(\mathcal{O})>1-\eta$, $\bb{P}$-a.s. Considering $\mathcal{C}$, if $x_0 \notin \mathcal{C}$ then
  again by choosing $j$ large enough, we have $\pi^0_{t_j}(\mathcal{C})<\eta$, $\bb{P}$-a.s. But
  this is a contradiction. Indeed, as
  $\pi^0_{t_j}(\mathcal{O})=\pi^0_{t_j} (\mathcal{C})+ \pi^0_{t_j}(\mathcal{O}/\mathcal{C})$ and
  $\pi^0_{t_j}(\mathcal{O})< \eta +\frac{1}{2} <1-\eta$, $\bb{P}$-a.s. Therefore, $x_0 \in \mathcal{C}$.
   
  We note that we have invoked the almost sure convergence only a finitely many number of times. This allows us to conclude that $\lim_{j\to\infty}\pi^0_{t_j}(A)=0$, $\bb{P}$-a.s. Since the sequence $t_j$ is arbitrary, we have $\plim_{t\to\infty}\pi^0_t(A)=0$.}
\end{proof}

\subsection{Stability of the filter}
We need the following lemma in proving the filter stability.
\begin{lem}{\cite[Pg. 55]{davidwilliams1991probability}}[\text{Scheffe's
		Lemma}]\label{sche}
	Suppose $f_n$ and $f$ are non-negative integrable functions in
	$\mathcal{L}^1(\Omega,\mathcal{B}, m)$ and
	$f_n\xrightarrow{n\to \infty} f \;\;\text{a.s}$. And also, suppose
	that $m(f_n)\xrightarrow{n\to \infty} m(f)$. Then
	$m\left(\left|f_n-f\right|\right)\xrightarrow{n\to \infty}0$
\end{lem} 
\begin{thm}\label{stab}
  Under the hypotheses of Theorem~\ref{th1}, If $\mu$ and $\nu$ are equivalent, then for
  any bounded continuous $g:X\rightarrow \bb{R}$,
  \begin{align}\nonumber
    \lim_{t\to\infty} \bb{E}\left[\left|\pi_t(g)-\bar{\pi}_t(g)\right|\right]=0\
  \end{align}.   
\end{thm}
\begin{proof} \editc{Firstly, note that for $J:=\frac{d\mu}{d\nu}$, the martingale convergence theorem implies that  
	\begin{align}\label{eq:martconvth}
\lim_{t\to\infty}\bb{E}\left[J(x_0)|\mathcal{F}^y_t\right]= \bb{E}\left[J(x_0)|\mathcal{F}^y_\infty\right],\text{ $\bb{P}$-a.s. and in $L^1$.}
	\end{align}
  From the Proposition~\eqref{pr1}, for any measurable $A\in \bb{B}(X)$
  \begin{align}\label{meas}
    \pi^0_\infty(A):=\plim_{t\to\infty}\pi^0_t (A)=\begin{cases}1,&    x_0 \in A\\
                 0, &x_0\notin A   
            \end{cases}
  \end{align}
  This is by definition the Dirac measure at $x_0$ and therefore, we have
  \begin{align*} 
  \plim_{t\to\infty}\bb{E}\left[J(x_0)|\mathcal{F}^y_t\right]=J(x_0).
   \end{align*}
   Now, we can choose a sequence $t_j$ such that 
   \begin{align*}
   \lim_{j\to\infty}\bb{E}\left[J(x_0)|\mathcal{F}^y_{t_j}\right]=J(x_0), \text{ $\bb{P}$-a.s.}
   \end{align*} 
   This implies that 
   \begin{align}\label{mctconseq}
   \bb{E}\left[J(x_0)|\mathcal{F}^y_\infty\right]=J(x_0),\text{ $\bb{P}$-a.s.}
   \end{align}
   Indeed, consider the limit in~\eqref{eq:martconvth} over the sequence $t_j$. To summarize, we have shown that 
   \begin{align*}
   \lim_{t\to\infty}\bb{E}\left[J(x_0)|\mathcal{F}^y_t\right]= \bb{E}\left[J(x_0)|\mathcal{F}^y_\infty\right]=J(x_0),\text{ $\bb{P}$-a.s. and in $L^1$}.
   \end{align*}
   With $M:=\underset{x\in X}{\editc{\sup}}\left|g(x)\right|<\infty$, we have
  \begin{align}\nonumber
    \bb{E}\left[\left|\pi_t(g)-\bar{\pi}_t(g)\right|\right]
    &=\bb{E}\left[\frac{\left|\mathbb{E}\left[g\left(\phi_t(x)\right)\left(\mathbb{E}\left[J(x_0)|\mathcal{F}^y_t\right]-J(x_0)\right)|\mathcal{F}^y_t\right]\right|}{\mathbb{E}\left[J(x_0)|\mathcal{F}^y_t\right]}\right]\\\nonumber
    &\leq \bb{E}\left[\frac{\mathbb{E}\left[\left|g\left(\phi_t(x)\right)\left(\mathbb{E}\left[J(x_0)|\mathcal{F}^y_t\right]-J(x_0)\right)\right||\mathcal{F}^y_t\right]}{\mathbb{E}\left[J(x_0)||\mathcal{F}^y_t\right]}\right]\\\nonumber
    &\leq M\bb{E}\left[ \frac{\mathbb{E}\left[\left|\mathbb{E}\left[J(x_0)|\mathcal{F}^y_t\right]-J(x_0)\right||\mathcal{F}^y_t\right]}{\mathbb{E}\left[J(x_0)|\mathcal{F}^y_t\right]}\right]\\\nonumber
    &\leq M\bb{E}\left[ \mathbb{E}\left[\frac{\left|\mathbb{E}\left[J(x_0)|\mathcal{F}^y_t\right]-J(x_0)\right|}{\mathbb{E}\left[J(x_0)|\mathcal{F}^y_t\right]}\Big|\mathcal{F}^y_t\right]\right], \text{ since $\mathbb{E}\left[J(x_0)|\mathcal{F}^y_t\right]^{-1}$ is $\mathcal{F}^y_t$-measurable}\\\label{final}
    &\leq M \bb{E}\left[ \frac{\left|\mathbb{E}\left[J(x_0)|\mathcal{F}^y_t\right]-J(x_0)\right|}{\mathbb{E}\left[J(x_0)|\mathcal{F}^y_t\right]}\right]
  \end{align}
  
  Finally, choose a subsequence $t_n\uparrow \infty$. Apply the
  Lemma~\ref{sche} for
  $f_n:=\frac{J(x_0)}{\mathbb{E}\left[J(x_0)|\mathcal{F}^y_{t_n}\right]}
  $ (Note that $J(x_0)>0\;\; \text{a.s}$) and $f:= 1$, to get the
  desired result.}
\end{proof}

\begin{remark}\label{r2}
\editc{  We note that 
  Assumptions~\ref{a1},
  \eqref{a2} and \eqref{a3} together form a sufficient condition for the
  notion of observability defined in
  \cite[Definition~2]{van2009observability}. This can be seen as follows:}

  Using~\eqref{mctconseq}, we can conclude that 
  $x_0$ is measurable
  with respect to $\mathcal{F}^y_\infty$. It implies that there exists
  a function $F$, that is measurable with respect to
  $\mathcal{F}^y_\infty$ such that
  $F: C\left(\left[0,\infty\right),\bb{R}^n\right)\rightarrow X$ and
  $x_0=F\left(Y_{[0,\infty)}\right)$. Therefore, we arrive at the
  conclusion that law of observation process determines the law of
  $x_0$ uniquely which is exactly the definition of observability in
  \cite{van2009observability}.
\end{remark}

\subsection{Comparison with the results in
  \cite{cerou2000long}}\label{compres}
\editc{Since we have used the techniques of \cite{cerou2000long} in proving Theorem~\ref{th1}, the natural question to ask is whether Theorem~\ref{th1} directly follows from \cite{cerou2000long}. In the following, we show that the assumptions in \cite{cerou2000long} are too restrictive to obtain the desired results  even for very simple systems.}
In~\cite{cerou2000long}, the
signal space is $X = \mathbb{R}^p$ while $h$ and $\phi_t$ are assumed
to satisfy the following assumption: for $x,y\in \mathbb{R}^p$,
\begin{align}\label{cerouassump}
  V_t\|x-y\|^2 \leq \int_0^t \|h(s,\phi_s(x))-h(s,\phi_s(y))\|^2 ds
  \leq RV_t\|x-y\|^2,
\end{align}
where, $V_t$ is an positive function such that $V_t \to \infty$ as
$t \to \infty$ and $R > 1$. Under this assumption, Proposition~$2.1$
of~\cite{cerou2000long} proves exactly the same conclusion as our
Theorem~\ref{th1}, namely, the limit in~\eqref{th1eq} showing the
concentration of the smoother. We show below, with an example, that
the Assumption~\ref{cerouassump} from~\cite{cerou2000long} is a
stronger assumption compared to the assumptions we use, in particular,
focusing on Assumptions~\ref{a1} and its consequence
in~\ref{uplowbound}.

To see this, consider $h(t,x)= G(t) x$, where $G(t)$ is a real valued
function. Equation~\eqref{cerouassump} becomes
\begin{align}\label{ceroulorenz}
  V_t\|x-y\|^2 \leq \int_0^t |G(s)|^2\|\phi_s(x)-\phi_s(y)\|^2 ds \leq
  RV_t\|x-y\|^2,
\end{align}
Now let $\phi_t$ be the solution of Equation~\eqref{lor1} (for
example, the Lorenz 63 or Lorenz 96 models used in
Section~\ref{examplenoncompact}). If $x \in U$, then
\begin{align}\label{uplowlipshitz}
  \exp(4HR_Ut-\|A\|t)\|x-y\|\leq \|\phi_t(x)-\phi_t(y)\|\leq
  \exp(\frac{4(H\bar{R})^2t}{\lambda})\|x-y\|.
\end{align}
Here, $4HR_U-\|A\|\leq 0$. Indeed, choosing $y\in U$ results in the
following:
\begin{align*}
  \exp(4HR_Ut-\|A\|t)\|x-y\| \leq \|\phi_t(x)-\phi_t(y)\|\leq diam(U) \,.
\end{align*} 
Due to the difference in exponents in upper and lower bounds of
Equation~\eqref{uplowlipshitz}, the condition~\eqref{ceroulorenz} is
not satisfied.

One can also try to find better (than Equation~\eqref{uplowlipshitz})
bounds of $\|\phi_t(x)-\phi_t(y)\|$. But the main issue is that the
continuity in $x$ of the flow $\phi_t(x)$ is not \emph{uniform in
  time}. Thus even though $\phi_t(x)$ is bounded uniformly in $t$,
i.e., $\|\phi_t(x)-\phi_t(y)\| \leq K$ where $K$ is the diameter of
$U$ (see Assumptions~\ref{a2}), it is not true that
\begin{align*}
  \|\phi_t(x)-\phi_t(y)\| \leq K_t \|x-y\|, \textrm{ for $K_t$
  uniformly bounded (in $t$).}
\end{align*}
Also, note that the lower bound in~\eqref{cerouassump} is also a
problem: for example, when the dynamics is dissipative, \emph{i.e.,}
$\nabla.F < 0$, where $F$ is the vector field of~\eqref{lor1}, then in
any ball (say, of radius $r$) $B_r(x)$ around $x \in U$, there is a
$y \in B_r(x)$ different from $x$ such that the following holds:
\begin{align*}
\|\phi_t(x) - \phi_t(y)\| \to 0 \textrm{ as $t \to \infty$}.
\end{align*}
In conclusion, it is clear that such models do not satisfy
Equation~\eqref{cerouassump} even when $h(t,x)=G(t)x$. But as we show
in Section~\ref{examplenoncompact}, these models do satisfy our
assumptions, in particular~\eqref{a1}.

The key difference between our observability assumption and that
of~\cite{cerou2000long} is that we use~\eqref{3} in
Assumption~\ref{a1}, instead of~\eqref{cerouassump}. Notice that
Equation~\eqref{3} involves $\phi_{s-t}$ with $t\leq s\leq t+\tau$,
and continuity in $x$ of $\phi_{s-t}(x)$ is uniform in $s-t$ for
$0\leq s-t\leq \tau$. Subsequently, we get Equation~\eqref{uplowbound}
which is crucial for our analysis. We also notice that the
bounds~\eqref{cerouassump} are replaced in our work by those
in~\eqref{uplowbound} in terms of $D_N$ defined in~\eqref{DN} and
$d_N$ defined in~\eqref{dN}, which occur naturally in dynamical
systems theory, allowing us to use their properties (along with other
Assumptions~\ref{a2},~\ref{a3},~\ref{a22} and~\ref{a4}) to
help us handle the case when the dynamics may be chaotic.

\section{Discrete time nonlinear filter}\label{S4}
\editc{To study the stability of the filter in discrete time,} we will \editc{set up} the discrete time filter in the
form where the filter at any time instant depends on the entire
observation sequence up to that instant. This form of the filter can
be easily converted (using an appropriate transformation of
observations) to the recursive form of the filter that is commonly
used in applications. We use the setup below in order to keep the
notation entirely parallel to the continuous time case we have
discussed until now.
\subsection{Setup}
Again, let the state space $X$ be $p$-dimensional complete Riemannian
manifold with metric $d$. On $X$, we have a homeomorphism
$T: X \rightarrow X$ along with initial condition $x_0$, whose
distribution is $\mu$. We denote discrete time with $k$. These
dynamics are observed partially in the following way.
\begin{align}\nonumber
  Y_k&= \sum_{i=1}^k h\left(i,T^i(x_0)\right)+W_k,
\end{align}
where, $h:\bb{Z}^+\times X\rightarrow \bb{R}^n$ and
$Y_k\in \mathbb{R}^n$ is the observation process and
$W_k\in \mathbb{R}^n$ is the position of an i.i.d random walk with
standard Gaussian increment after $k$ steps, starting at
origin. Moreover, $x_0$ and $W_{k+1}-W_{k}$ are assumed to be
independent for any $k\geq 1$. Therefore,
\begin{align*}
  \left\{X\times \left(\bb{R}^n\right)^{\bb{Z}^+},\bb{B}(X)\otimes \bb{B}\left((\bb{R}^n)^{\bb{Z}^+}\right),\bb{P}=\mu\otimes \bb{P}_W\right\}
\end{align*} 
is considered to be our probability space. Here, $\bb{B}(.)$ denotes
the borel $\sigma$-algebra of the corresponding space and $\bb{P}_W$
is the probability measure of $W$. Let
$\mathcal{F}^{y}_k=\sigma\left\{Y_i:0\leq s\leq k,\; i\in
  \bb{Z}^+\right\}$, the observation process filtration. We shall see
that the results of stability for the case of continuous time extend
to the discrete time case with very minor changes. Noting this, we
denote all the quantities that appear in both continuous and discrete
time cases by same symbols.
\begin{note}
  $\pi^0_k,\; \pi_k$ and $\bar{\pi}_k$ have similar meanings to what
  they mean in continuous time case.
\end{note}
Define,
\begin{align}\nonumber
  Z(k,x,Y_{0:k}):= \exp\left(\sum_{i=1}^k h\left(i,T^i(x)\right)^T\left(Y_i-Y_{i-1}\right)-\frac{1}{2}\sum_{i=1}^k \left\|h\left(i,T^i(x)\right)\right\|^2 \right),
\end{align}
with the convention that $\sum_{1}^0:=0$. From Bayes' rule, for any
bounded continuous function $g$,
\begin{align}\label{dinit}
  \pi^0_k(g)=\mathbb{E}\left[g(x_0)|\mathcal{F}^y_k\right]=\frac{\int_X g(x)Z(k,x,Y_{0:k})\mu(dx)}{\int_{X} Z(k,x,Y_{0:k})\mu(dx)}
\end{align}
For a fixed $k$, the filter is given by
\begin{align}\nonumber
  {\pi}_k(g)=\mathbb{E}\left[g(T^k(x_0))|\mathcal{F}^y_k\right]=    \frac{\int_{X}g(T^k(x))Z(k,x,Y_{0:k})\mu(dx)}{\int_{X}Z(k,x,Y_{0:k})\mu(dx)}
\end{align}
Choosing an incorrect initial condition with law $\nu$, expression for
the corresponding incorrect filter is given by
\begin{align}\label{dfil}
  \bar{\pi}_k(g)= \frac{\int_X g\left(T^k(x)\right)Z(k,x,Y_{0:k})\nu(dx)}{\int_{X}Z(k,x,Y_{0:k})\nu(dx)}
\end{align}

\subsection{Stability of the filter}
In the discrete time case, as earlier, stability of the filter is achieved if we
show that, for any bounded continuous $g:X\rightarrow \mathbb{R}$,
\begin{align}\nonumber
  \lim_{k\to \infty} \mathbb{E}[\left|\pi_k(g)-\bar{\pi}_k(g)\right|]=0
\end{align}

To establish the above, we need a discrete analog of
Theorem~\ref{th1}. This can be done under the following discrete
analogs of Assumptions~\ref{a1},~\ref{a2},~\ref{a3}. Again note that
we use same symbols for the quantities that appear in both the cases.
\begin{asu}\label{ad2}
	There exists a bounded open set $U$ such that
	$\overline{T U}\subset U $.
\end{asu}
\begin{asu}\label{ad3}
	$\forall x,\;y\in U$, we have  $d(T x,T y)\leq C d(x,y)$, for some $C>1$.
\end{asu}
\begin{asu}\label{ad1}
  There exists $\rho_k,\;R,\;k_0>0$ such that $\forall x_1,x_2\in U$
  \begin{align}\label{d3}
    \forall k\geq 0,\; \rho_k d(x_1,x_2)^2\leq \sum_{i=k}^{k+k_0} \left\|h\left(i,T^{i-k}(x_1)\right)-h\left(i,T^{i-k}(x_2)\right)\right\|^2  \leq R\rho_k d(x_1,x_2)^2,
  \end{align} 
  where, $\rho_k$ is a positive non-decreasing function such that
  $\lim_{k\to\infty}\frac{\sum_{i=0}^k \rho_{i}}{\rho_k}=\infty$,
  $\frac{\editc{k\rho_k}}{\sum_{i=0}^k \rho_{i}}\leq C'$ (for some $C'>0$) and
  $R>1$.
\end{asu}

\begin{asu}\label{ad22}
\editc{  For $\mathcal{V} \subset U \times U$ , where  $(U\times U)\backslash \mathcal{V}$ is a $\sigma$-null measure set, and for $(x, y) \in \mathcal{V}$, satisfying $d(x, y) \geq b > 0$, the following holds}
  \begin{align*}
D^2_N(x,y) \geq L^2(b)\sum_{i=0}^N\rho_{i\tau},
\end{align*}   
where, $L(b)$ is a positive constant.
\end{asu}
\begin{asu}\label{ad4}
$supp(\mu)\subset U$	
\end{asu}
It follows from Assumption~\ref{ad1} that  
\begin{align}\label{duplowbound}
  \sum_{i=1}^N \rho_{ik_0} d\left(T^i(x),T^i(y)\right)^2\leq \sum_{i=0}^{k} \left\|h\left(T^i(x)\right)-h\left(T^i(y)\right)\right\|^2 \leq R\sum_{i=1}^{N+1}\rho_{ik_0} d\left(T^i(x),T^i(y)\right)^2,\;\; \forall x,y\in U, 
\end{align}
where, $N=\lfloor\frac{k}{k_0}\rfloor$.
\begin{remark}
  The significance of the above assumptions is exactly the same as
  that of the assumptions in Section~\ref{S3}.
\end{remark}
Now we state the discrete analogs of Theorem~\ref{th1},
Proposition~\ref{pr1} and Theorem~\ref{stab}.
\begin{thm}\label{thd1}
  Suppose $\mu$ is absolutely continuous with respect to volume, $\sigma$
  of $X$ and $\frac{d\mu}{d\sigma}$ is continuous on the support of
  $\mu$. Under the Assumptions~\ref{ad2}---\ref{ad4},
  \begin{align*}
    \editc{\plim_{k\to\infty}}e^{\alpha k}\left(\pi^0_k \left(\left\{x \in X : d(x,x_0)\leq a\right\}\right)-1\right)=    0, \;\; \forall a>0,\;\;
  \end{align*}
  and for some $\alpha:=\alpha(a)>0$ which depends only on $a$.
\end{thm}
\begin{proof}
  The proof of this theorem follows exactly in the same lines as that
  of Theorem~\ref{th1}. So the proof is omitted.
\end{proof}
\begin{prop}\label{prd1}
  Under the hypotheses of Theorem~\ref{thd1},
  \begin{align*}
    \editc{\plim_{k\to \infty}}\pi^0_k (A)= 0,\;\; \forall A \in\bb{B}(X),\;\;x_0\notin A
  \end{align*}
\end{prop}
\begin{proof}
  We observe that the proof of Proposition~\ref{pr1} remains unchanged
  if the continuous time is replaced with discrete time.
\end{proof}
\begin{thm}\label{stabd}
  Under the hypotheses of Theorem~\ref{th1}, If $\mu$ and $\nu$ are equivalent, then for
  any bounded continuous $g:X\rightarrow \bb{R}$,
\begin{align}\nonumber
\lim_{k\to\infty} \bb{E}\left[\left|\pi_k(g)-\bar{\pi}_k(g)\right|\right]=0\
\end{align}.   
\end{thm}
\begin{proof}
  Proof is again omitted as it is exactly in the same lines as that of
  Theorem~\ref{stab}.
\end{proof}
\begin{remark}\label{r3}
  Remarks analogous to Remark~\ref{r2} and the rest of the remarks of
  Section~\ref{S3} follow in the case of discrete time.
\end{remark}

\section{Structure of the conditional distribution}\label{S5}

In this section, we will see that the conditional distribution of $x_t$ after
large times puts most of its mass on the topological attractor.
We restrict ourselves to the case of continuous time filter (similar
conclusions can be drawn for discrete time case as well). 
Recall that topological attractor $\Lambda$ is defined
(e.g. \cite[Pg. 128]{katok1996introduction}) as
\begin{align*}
\Lambda:= \cap_{t\geq 0} \phi_t(U),
\end{align*}
where $U$ is an open set such that $\overline{\phi_t(U)}\subset U$,
for $t>0$, as introduced in Assumption~\ref{a2}. We make a further
assumption:
\begin{asu}\label{alast}
  $\forall x\in X$, there exists $t(x)\geq 0$ given by
  $t(x):= \inf\{t\geq 0: \phi_{t}(x)\in U\}$.
\end{asu}
\begin{thm}
  Under the Assumption~\ref{alast},
  \begin{align*}
    \lim_{t\to \infty}	\pi_t(\Lambda_{s})=1,\; \forall s\geq 0,
  \end{align*}
  where, 
  \begin{align*}
    \Lambda_s:=\cap_{0\leq r\leq s} \phi_r(U),
  \end{align*}
  for $s\geq 0$.
\end{thm}

\begin{proof}
  From \eqref{smoothcontdist}, for any $A\in \bb{B}(X)$, we have
  \begin{align*}
    {\pi}^0_t(A)
    &=\mathbb{E}[\mathbbm{1}_{\{x_0\in A\}}|\mathcal{F}^y_t]\\
    &= \frac{\int_A Z \left(t, x, Y_{\left[0,t\right]}
      \right) \mu \left(dx \right)}{\int_{X} Z \left(t, x,
      Y_{\left[0,t\right]} \right) \mu \left( dx \right)}
  \end{align*}

  From \eqref{filcontdist}, for any $A\in \bb{B}(X)$, we have
  \begin{align}\label{support}
    {\pi}_t(A) = \mathbb{E}\left[\mathbbm{1}_{\{\phi_t(x_0)\in
    A\}}|\mathcal{F}^y_t\right]
    & = \frac{\int_{\{\phi_t(x)\in
      A\}}Z(t,x,Y_{[0,t]})\mu(dx)}{\int_{X}Z(t,x,Y_{[0,t]})\mu(dx)}\\
    &=\frac{\int_{A}Z(t,\phi_{-t}(y),Y_{[0,t]})\mu\circ
      \phi_{-t}(dy)}{\int_{X}Z(t,\phi_{-t}(y),Y_{[0,t]})\mu\circ
      \phi_{-t}(dy)}
  \end{align}
  Therefore, support of ${\pi}_t$ is always contained in the support
  of $\mu\circ \phi_{-t}$. So, it is sufficient to show that
  asymptotically the support of $\mu\circ \phi_{-t}$ is near the
  topological attractor (\emph{i.e., $\Lambda_{s}$}) to conclude that
  after large times, ${\pi}_t$ puts negligible mass far away from the
  topological attractor.

  To that end, we define the following disjoint family of sets,
  $\{U^s_r\}_{r\geq 0}$, for a given $m$:
  \begin{align*}
    U_r^s:=\left\{x\in X: \inf\left\{t\geq 0: \phi_t(x) \in
    \Lambda_s\right\}=r\right\}.
  \end{align*}
  From the Assumption~\ref{alast}, for any given $s\geq 0$, it
  follows that
  \begin{align*}
    X=\cup_{r\geq 0} U_r^s
  \end{align*}
  Now, for a given $s\geq 0$ and $t\geq s$, consider
  \begin{align*}
    \mu\circ \phi_{-t}(\Lambda_{s})
    &=\mu\left(\left\{ x\in X: \phi_t(x)\in \Lambda_s\right\}\right)\\
    &=\mu\left(\left\{ x\in X: \inf\left\{r\geq 0:\phi_r(x)\in
      \Lambda_s\right\}\leq t \right\}\right)\\
    &=\mu\left(\cup_{0\leq r\leq t} U_r^s\right)\\
  \end{align*}
  From above, we have
  $\lim_{t\to\infty}\mu\circ \phi_{-t}(\Lambda_{s})=1,\;\forall s\geq
  0$. Note that this is not a uniform limit in $s\geq 0$. This
  concludes that asymptotically ${\pi}_t$ is supported on $\Lambda_s$
  for every $s\geq 0$.
\end{proof}

\begin{remark}\label{remsupp}
  If $\mu$ \editc{has a} bounded support, then following the computations
  above, we can conclude that $\mu\circ \phi_{-t}$ is supported on $U$
  after some finite time. To see this, note that $\Lambda_0=U$ and let
  $S:= \inf\{s\geq 0: supp(\mu)\subset \cup_{0\leq r\leq
    s}U_r^{0}\}$. Now, we have
  \begin{align*}
    \mu\circ \phi_{-S}(U)
    &=\mu\left(\left\{ x\in X:\phi_{S}(x)\in U\right\}\right)\\
    &=\mu\left(\left\{ x\in X: \inf\left\{r\geq 0:\phi_r(x)\in
      U\right\}\leq S \right\}\right)\\
    &=\mu\left(\cup_{0\leq r\leq S} U_r^0\right)\\
    &=1.
  \end{align*}
  Therefore, $\pi_t$ is also supported on $U$ after some finite time
  $S$.
\end{remark}

\begin{remark}
  In the above computations, it is clear that $\mu$ can be replaced by
  $\nu$ (or any other probability measure) to arrive at similar
  conclusions.
\end{remark}

\begin{remark}\label{initisupp}
  To summarize, under the Assumption~\ref{alast}, any probability
  measure $m$ (with bounded support) evolved under the flow
  $\{\phi_t\}_{t\geq 0}$ is supported entirely on $U$, after some
  finite time. In practice, the system of interest would have already
  been evolved for long time before we started observing the system
  and many systems of interest satisfy
  Assumption~\ref{alast}. Therefore, it is reasonable to have
  Assumption~\ref{a4}.
\end{remark}

\section{Examples and Discussions}\label{S7}
\subsection{Examples with compact state space}\label{examplecompact}
We consider $(X,d)$ to be compact and
$h(.,.):\bb{R}^+\times X\rightarrow \bb{R}^p$ is such that $h(t,.)$ is
bi-Lipshitz for every $t\geq 0$ that satisfies the following:
\begin{align*}
  K(t)d(x,y)\leq \|h(t,x)-h(t,y)\|\leq R K(t) d(x,y),
\end{align*}  
for some $\alpha>0$,$R>1$, $K(t)$ such that $K(t)=O(t^\alpha)$ and is
increasing in $t$. Since any dynamical system
$\{\phi_t\}_{t\in \bb{R}}$ with $\phi_t$ being a $C^{1+\alpha}$
diffeomorphism on $X$ (with $\alpha>0$, for every $t \in \bb{R}$) is
such that $\phi_t$ is bi-Lipshitz, we have
\begin{align*}
  \frac{1}{MC^t}d(x,y)\leq d(\phi_tx,\phi_ty)\leq MC^t d(x,y),
\end{align*}
$\forall t\in \bb{R}$ and for some $C,M>1$. Now consider the following
expression:
\begin{align*}
  \int_t^{t+\tau} \left\|h\left(s,\phi_{s-t}(x_1)\right)-h\left(s,\phi_{s-t}(x_2)\right)\right\|^2 ds
\end{align*}
From the above, we have
\begin{align*}
  \int_t^{t+\tau} \left\|h\left(s,\phi_{s-t}(x_1)\right)-h\left(s,\phi_{s-t}(x_2)\right)\right\|^2 ds &\leq \int_t^{t+\tau} R^2K^2(s)d\left(\phi_{s-t}(x_1),\phi_{s-t}(x_2)\right)^2 ds\\
&\leq M^2R^2 d(x_1,x_2)^2 \int_t^{t+\tau} K^2(s)C^{2(s-t)} ds
\end{align*}
Similarly we can obtain the following lower bound:
\begin{align*}
  \int_t^{t+\tau} \left\|h\left(s,\phi_{s-t}(x_1)\right)-h\left(s,\phi_{s-t}(x_2)\right)\right\|^2 ds \geq  \frac{1}{M^2} d(x_1,x_2)^2 \int_t^{t+\tau} K^2(s)C^{-2(s-t)} ds
\end{align*}
We consider $K(t)$ to be of the form $=Bt^q$, for some $q>0$. Define
$\rho^1_t:= B^2\int_t^{t+\tau} t^{2q}C^{-2(s-t)} ds$ and
$\rho^2_t:= B^2\int_t^{t+\tau} t^{2q}C^{2(s-t)} ds$. It can be seen
from computing the integrals that
\begin{align*}
  1\leq \frac{\rho^2_t}{\rho^1_t}\leq \bar{M},
\end{align*}  
for some $\bar{M}>1$ independent of $t\geq 0$. It can be seen that
$\rho^1_t\sim O(t^{2q})$. Therefore, by defining $\rho_t$ in
Assumption~\ref{a1} as $\rho_t:=\frac{1}{M}\rho^1_t$, we can conclude
that the above model satisfies both Assumptions~\ref{a1} and
~\ref{a3}. Since $X$ is compact, Assumptions~\ref{a2} hold trivially
by choosing $U$ in Assumption~\ref{a2} as $X$. In the above, we
presented only continuous time models. Models in discrete time can be
constructed similarly.
 
In the following, we give sufficient conditions for
Assumption~\ref{a22} to hold. Recall that Assumption~\ref{a22} says
that there is a set $\mathcal{V}\subset X\times X$ that is of full
measure under $\sigma\otimes\sigma$ such that for $x,y \in \mathcal{V}$
satisfying $d(x,y)\geq b>0$, the following holds
\begin{align}\label{aexp}
  D^2_N(x,y) \geq L^2(b)\sum_{i=0}^N\rho_{i\tau},
\end{align}   
where, $L(b)$ is a positive constant. In the following, we show that
\ref{aexp} holds for a particular type of dynamical systems
\emph{viz.,} uniformly hyperbolic systems
\cite[Definition~4.1]{shub1986global}. The arguments made are
independent of whether time is discrete or continuous. So without loss
in generality, let us suppose that the time is discrete with $T$ being
the homeomorphism.  \setcounter{footnote}{0} Suppose $T$ is a
$C^{1+\alpha}$ uniformly hyperbolic diffeomorphism with
$\alpha>0$. From \cite[Proposition~7.4]{shub1986global}, $T$ is
expansive, \emph{i.e.,} there exists $\epsilon>0$ such that for every
$x,y \in X$ with $x\neq y$, there exists $n \in \bb{Z}$ such that
$d(T^nx,T^ny)>2\epsilon$ (for the clarity in expressions, we write $T^nx$ for $T^n(x)$ in this section). From the continuity of $T$ and compactness
of $X$, we have the following lemma whose proof is provided below for sake of completeness (see \cite{urlunif}):
\begin{lem}\label{uniexp}
  For any $\delta >0$ and for some $\epsilon>0$ \editc{(independent of $\delta$)}, if $x,y \in X$ such
  that $d(x,y)\geq\delta$ then there exists $J\in \bb{N}$ \editc{(independent of $x$ and $y$)} such that
  for some $n \in \bb{Z}$ with $|n|\leq J$, we have
  \begin{align*}
    d(T^nx,T^ny)>\epsilon
  \end{align*}
\end{lem}
\begin{proof}
  Consider the compact set,
  $K:=\{z=(x,y)\in X\times X: d(x,y)\geq\delta\}$. Choose $x,y\in X$
  such that $d(x,y)\geq \delta$. From expansivity, there exists
  $n(x,y)\in \bb{Z}$ such that
  $d(T^{n(x,y)}x,T^{n(x,y)}y)>\epsilon$. Define,
  $G(.,.): X\times X\rightarrow X\times X$ by
  $G(u,v)=(T^{n(x,y)}u,T^{n(x,y)}v)$. It is clear that $G$ is
  continuous on $X\times X$ and from the continuity of $G$, there is a
  neighbourhood $U({\bar{z}})$ around $\bar{z}=(x,y)$ such that
  $d(T^{n(x,y)}u,T^{n(x,y)}v)>\epsilon,\; \forall (u,v)\in
  U({\bar{z}})$. Since $\bar{z}=(x,y)$ is an arbitrary point in $K$,
  we can cover $K$ by a family of open sets given by
  $\{ U(z)\}_{z\in K}$. From compactness of $K$, there is a finite set
  $\{z_i\}_{i=1}^{k_0}\subset K$ such that
  $K\subset \cup _{i=1}^{k_0} U({z_i})$. Now, defining
  \begin{align*}
    J:= \max _{i=1,..,k_0} \left(|n(x_i,y_i)|:z_i=(x_i,y_i) \right),
  \end{align*} 
  we have the result.
\end{proof}
In particular, if we choose $\delta<\epsilon$, $d(T^nx,T^ny)>\epsilon$
for infinitely many $n\in \bb{Z}$. Suppose, $x$ is in the global
unstable manifold of $y$ such that $d(x,y)>\epsilon$, \emph{i.e.,}
\begin{align*}
  d(T^nx,T^ny)\leq B \lambda^{n}d(x,y),
\end{align*}
where, $n\leq 0$, $B>0$ and $\lambda >1$ (independent of $x$ and
$y$). It is clear that there exists $\bar{N}$ such that
$d(T^nx,T^ny)<\epsilon,\; \forall n\leq -\bar{N}$. Therefore, from the
above lemma, it is clear that if $|n|>\bar{N}$ and
$d(T^nx,T^ny)\geq\epsilon$ then $n>0$. Let $\{n_k(x,y)\}_{k\in\bb{N}}$
be a subsequence such that
$d(T^{n_k(x,y)}x,T^{n_k(x,y)}y)\geq\epsilon$. From the above
discussion, it is clear that $\{n_k(x,y)\}_{k\in\bb{N}}$ is an
infinite set and in particular, $n_k(x,y)>\bar{N}$ infinitely many
times. Therefore, without loss in generality, let us restrict the
attention to $\{n_k(x,y)\}_{k\in\bb{N}}$ such that
$n_k(x,y)\geq \bar{N},\; \forall k\in \bb{N}$. From Lemma~\ref{uniexp}
and above discussion, we have the following:
\begin{align*}
  n_{k+1}(x,y)-n_k(x,y)\leq J.
\end{align*}
Note that $J$ is independent of $x$ and $y$ as long as
$d(x,y)\geq \epsilon$. Therefore, the cardinality of the set
$\{n_k(x,y) \}_{k\in \bb{N}} \cap [\bar{N}+1,2,3,...,\bar{N}+\hat{N}]$
is at least $\lfloor\frac{\hat{N}}{J}\rfloor$, for any
$\hat{N}\in \bb{N}$. As a result, we have the following for
$N> \bar{N}$:
\begin{align}\nonumber
  D^2_N(x,y) &\geq \epsilon\sum_{\substack{k\in \bb{N}, \\ \bar{N}<n_k(x,y)\leq N}}\rho_{n_k(x,y)\tau}+ \sum_{i=0}^{\bar{N}}d(T^ix,T^iy)\rho_{i\tau}
 \geq \epsilon \sum_{i=\bar{N}+1}^{\lfloor\frac{N-\bar{N}}{J}\rfloor +\bar{N}+1}\rho_{i\tau} +\sum_{i=0}^{\bar{N}}d(T^ix,T^iy)\rho_{i\tau}\\\nonumber
 &\geq \min\left(\epsilon, \inf_{\substack{x,y\in X, \\ d(x,y)>\epsilon}}\left(\min_{i\leq \bar{N}}\left(d(T^ix,T^iy)\right)\right)\right)\sum_{i=0}^{\lfloor\frac{N-\bar{N}}{J}\rfloor +\bar{N}+1}\rho_{i\tau}\\\label{i1}
 &\geq \min\left(\epsilon, \inf_{\substack{x,y\in X, \\ d(x,y)>\epsilon}}\left(\min_{i\leq \bar{N}}\left(d(T^ix,T^iy)\right)\right)\right)\sum_{i=0}^{\lfloor\frac{N}{J}\rfloor}\rho_{i\tau}\\\label{i2}
 &\geq G(J)\min\left(\epsilon, \inf_{\substack{x,y\in X, \\ d(x,y)>\epsilon}}\left(\min_{i\leq \bar{N}}\left(d(T^ix,T^iy)\right)\right)\right)\sum_{i=0}^{N}\rho_{i\tau},
\end{align}
where, $G\left( J\right)>0$ depends only on $J$. Inequality~\eqref{i1}
follows from non-decreasing property of $\rho_t$, applying the lowest
bound to any sum up to first $\lfloor\frac{N}{J}\rfloor$ terms of an
subsequence of a non-decreasing sequence and inequality~\eqref{i2}
follows from the form of $\rho_t$. And also, from uniform
hyperbolicity, bi-Lipshitz property of $T$ and $d(x,y)>\epsilon$, for
$n\leq \bar{N}$, we have
\begin{align*}
  d(T^nx,T^ny)&\geq \frac{1}{C^n}d(x,y)\\
              &\geq \frac{1}{C^{\bar{N}}} d(x,y)\\
              & >\frac{1}{C^{\bar{N}}}\epsilon,
\end{align*}
for some $C>1$. Therefore, we have
\begin{align*}
  \inf_{\substack{x,y\in X, \\ d(x,y)>\epsilon}}\left(\min_{i\leq \bar{N}}\left(d(T^ix,T^iy)\right)\right)> \frac{1}{C^{\bar{N}}}\epsilon
\end{align*}
and we have shown that if $x$ lies in the unstable manifold of $y$ and
$d(x,y)>\epsilon$, we have
\begin{align*}
  D^2_N(x,y)\geq \min\left(G\left(\epsilon, J\right), \frac{1}{C^{\bar{N}}}\epsilon \right)\sum_{i=0}^N\rho_{i\tau}
\end{align*}
Now, we extend the above inequality, to $x$ and $y$ when $x$ does not
lie in either global stable or unstable manifolds of $y$. To that end,
from \cite{barreira2013introduction}, it is known that global stable
manifolds form a foliation of $X$ and global unstable manifold through
a given point in $X$ is their transversal. Therefore, for a given $x$
and $y$ such that $d(x,y)>\epsilon$ and $x$ that does not lie in the
stable manifold of $y$, there is a point $z\in X$ contained in the
global unstable manifold of $y$ such that $x$ is the global stable
manifold of $z$ and we have
\begin{align*}
  d(T^nz,T^ny)\leq d(T^nz,T^nx)+d(T^nx,T^ny) 
\end{align*} 
From the property of global stable manifold and Lemma~\ref{uniexp},
there exists $J_1$ such that
$d(T^nz,T^nx)\leq \frac{\epsilon}{2},\; \forall n\geq J_1$. If
$J_1 > J$, we replace $J$ by $J_1$. Choosing $n=n_k(x,y)$, we get
\begin{align*}
  \epsilon<d(T^nz,T^ny)&\leq \frac{\epsilon}{2}+d(T^nx,T^ny) \\
                       &\frac{\epsilon}{2}<d(T^nx,T^ny).
\end{align*} 
Therefore, we have 
\begin{align*}
  D^2_N(x,y)\geq \min\left(G\left(\frac{\epsilon}{2}, J\right), \frac{1}{C^{\bar{N}}}\epsilon\right)\sum_{i=0}^N\rho_{i\tau}.
\end{align*}
Since the global stable manifold is strictly a lower dimensional
manifold due to uniform hyperbolicity, we proved that \eqref{aexp}
holds on a full measure set under measure $\sigma\otimes \sigma$ ($\sigma$ is
the Riemannian volume), which is sufficient for Theorem~\ref{th1} to
hold.

\subsection{Examples with non-compact state space}\label{examplenoncompact} 
We now consider $X=\bb{R}^p$ (which is non-compact) and continuous
time models only. Choose
$h(t,x):=K(t)\bar{h}(x):\bb{R}^+ \times X\rightarrow \bb{R}^p$ with
any bi-Lipshitz $\bar{h}:X\rightarrow \bb{R}^p$ and $K(t)=O(t^q)$ . In
the following, we show that the class of dynamical systems given by
\eqref{lor1} along with the chosen observation model satisfy
Assumptions~\ref{a1}, ~\ref{a2} and ~\ref{a3}. To that end, let
$\phi_t$ be the solution of the ordinary differential equation given
below 
\begin{align}\label{lor1}
  \frac{d}{dt}\phi_t +A\phi_t+B\left(\phi_t,\phi_t\right)=f,
\end{align}
where, $B(.,.):\bb{R}^p\times \bb{R}^p\rightarrow \bb{R}^p$ is
symmetric bi-linear operator such that
$u^TB(u,u)=0,\; \forall u\in \bb{R}^p$ and $A$ is $p\times p$ matrix
such that $u^TAu>\lambda \|u\|^2,\;\forall u\neq 0$. Observe that we
have $|u^T B(v,w)|\leq H\|u\|\|v\|\|w\|$, for some $H$. From
\cite[Remark~2.4]{kelly2014well}, we have the existence of bounded
open set $U$ such that $\overline{\phi_t U}\subset U$. And also, from
\cite[Lemma~2.6]{kelly2014well}, we have the following:
\begin{align}\label{forlip}
  \|\phi_\tau(u)-\phi_\tau(v)\|\leq e^{\gamma \tau}\|u-v\|,
\end{align}
$\forall u\in U$, $\forall v\in \bb{R}^p$ and for some $\gamma= \frac{4(H\bar{R})^2}{\lambda}>0$, where $\bar{R}$ is the least positive number such that $\|\phi_t(x)\|\leq R$, for any $x\in U$.
Defining, $e_t:= \phi_t(u)-\phi_t(v)$, we have
\begin{align*}
  \frac{d}{dt}e_t +Ae_t+B(\phi_t(u),\phi_t(u))-B(\phi_t(v),\phi_t(v))&=0\\
e^T_t \frac{d}{dt}e_t + e^T_tAe_t+e^T_t\left(B\left(\phi_t(u),\phi_t(u)\right)-B\left(\phi_t(v),\phi_t(u)\right)\right)&=0\\
 \frac{1}{2}\frac{d}{dt}\|e_t\|^2 + e^T_tAe_t+2 e^T_t\left(B\left(\phi_t(u),e_t\right)-B\left(e_t,e_t\right)\right)&= 0\\
\frac{1}{2} \frac{d}{dt}\|e_t\|^2 + \|A\|\|e_t\|^2-2H \|e_t\|^2 \|\phi_t(u)\|&\geq 0\\
 \frac{d}{dt}\|e_t\|^2 + \left(2\|A\|+4H R_U\right)\|e_t\|^2 &\geq 0,
\end{align*}
where, $R_U:=\sup_{u\in U}\sup_{t\geq 0} \|\phi_t(u)\|$ and we used the
properties of $A$ and $B(.,.)$. We integrate the above equation to
get,
\begin{align*}
  \|e_t\|^2\geq \|e_0\| - \left(2\|A\|+4H R_U\right)\int_0^t \|e_s\|^2 ds
\end{align*}
Applying the inequality from \cite[Lemma~2]{lowgorn}, we have
\begin{align}\label{baclip}
  \|\phi_\tau \left(u\right)-\phi_\tau\left(v\right)\|\geq \exp\left(-\left(\|A\|+2H R_U\right)\tau\right)\|u-v\|
\end{align}
From the above, it is clear that Assumptions~\ref{a2} and ~\ref{a3}
hold. From the calculations similar to those in
Section~\ref{examplecompact}, we can conclude that if $x_1, x_2\in U$, then Assumptions~\ref{a1} holds. In the above, we have shown that Assumption~\ref{a1} holds in a trivial case. Thus we only need to check the validity of assumption~\ref{a22} in this case. 

In the following, we discuss two well-known models, \emph{viz.,}
Lorenz 96 model and Lorenz 63 model and give numerical evidence that
these models satisfy Assumption~\ref{a22}. To that end, we will show
from numerical computations that
\begin{align} \label{numbound}
D_N(x,y):= \sum_{i=0}^{N}\rho_{i\tau}d(\phi_{i\tau}(x),\phi_{i\tau}(y))\geq
\bar{H}\sum_{i=0}^{N}\rho_{i\tau} \,,  
\end{align}
for some $\bar{H}>0$ and $x,y$ are such that
$d(x,y)>b$, for some $b>0$.

\paragraph*{\textbf{Lorenz 63 model\cite{lorenz1963deterministic}}}

In this case, $X=\bb{R}^3$  $\phi_t(u)=[x^1_t(u),x^2_t(u),x^3_t(u)]^T$ with $\phi_0(u)=u$ 
\begin{align}\label{lor63} 
\begin{split}
\frac{d}{dt}x^1_t&=a(x^2_t-x^1_t)\\
\frac{d}{dt}x^2_t&=x^1_t\left(b-x^3_t\right)-x^2_t\\
\frac{d}{dt}x^3_t&= x^1_tx^2_t-cx^3_t,
\end{split}
\end{align}
where, we dropped the dependence of $u$. For $a=10,\;b=28$ and
$c=\frac{8}{3}$, it is known that the above model exhibits chaotic
behavior. 
\begin{figure}[t!]
  \centering
  \includegraphics[width=\linewidth , height=6cm]{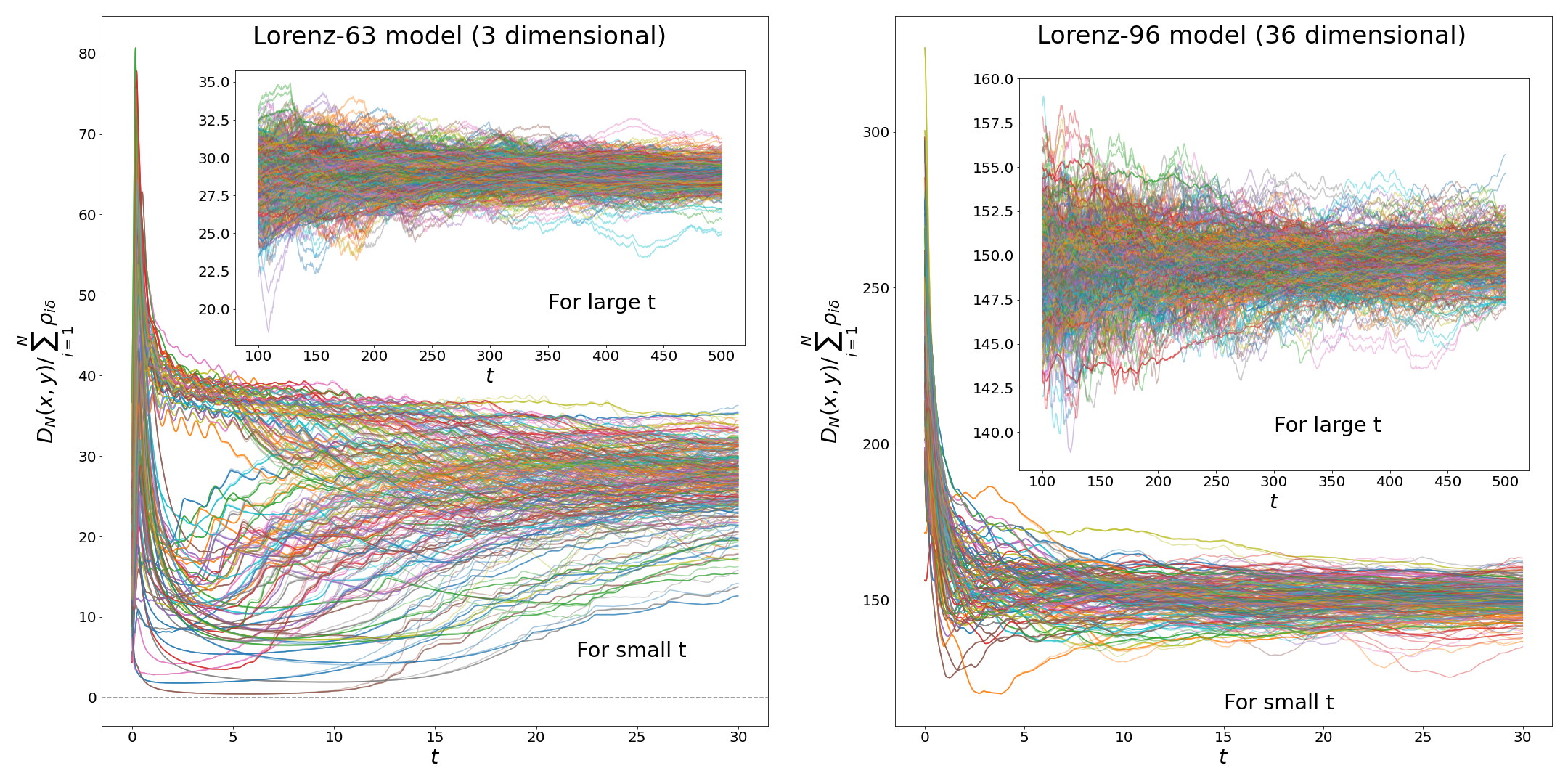}
  \caption{\editc{Dependence of
    $\frac{D_N(x,y)}{\sum_{i=0}^N \rho_{i\tau}}$ \emph{vs}
    $t=N\tau$ with $\tau = 0.01$ for $100$ samples. We have
    $t=N\tau$ with $\tau = 0.01$. We chose $100$ different pairs of $(x,y)$ for five different choices of $\rho_t=1000,\; t+1000,\; \log(t+1000); t^2+1000; t^3 + 1000$. The initial
    conditions for the samples are randomly chosen from uniform
    distribution on $[-10,10]^{p}$ where the dimension $p=3$ for Lorenz 63 model (left panel) and $p=36$ for the Lorenz 96 model (right panel). The insets show the plots for large $t$. (Note that the Lyapunov time scale for both these models is $O(1)$.)}}
    \label{fig:63}
\end{figure}

\paragraph*{\textbf{Lorenz 96 model\cite{lorenz1996predictability}}}
For this model, $X=\bb{R}^p$,
$\phi_t(u)=[x_t^1(u),x_t^2(u),...,x_t^p(u)]^T$ with
\begin{align*}
\frac{d}{dt}x_t^i&=(x_t^{i+1}-x_t^{i-2})x_t^{i-1}-x_t^{i} +F
\end{align*}
where, it is assumed that $x_t^{-1}=x_t^{p-1}$, $x_t^{0}=x_t^{p}$,
$x_t^{1}=x_t^{p+1}$ and we again dropped the dependence of $u$. For
$F=8$, this model is known to exhibit chaotic behavior. 

\editc{Note that for both the models (left and right panel of
  Figure~\ref{fig:63}), the plots with five vary different choices of
  $\rho_t=1000,\; t+1000,\; \log(t+1000),\; t^2+1000,\; t^3 + 1000$
  look very similar, and give a strong numerical evidence that indeed
  Equation~\eqref{numbound} is satisfied by both Lorenz 63 and Lorenz
  96 models. Providing an analytical proof of the validity of
  Assumption~\ref{a22} for these models or more generally for
  dynamical systems of the type given in~\eqref{lor1} is an
  interesting open question.}

\subsection{Qualitative understanding of Assumptions~\ref{a22} and~\ref{ad22}}\label{qual}
\editc{In the following we will argue that a system with sensitivity to initial conditions and a positive Lyapunov exponent satisfies the Assumptions~\ref{a22} and~\ref{ad22}}.
We
restrict ourselves to the discrete time setup and to that end, we
consider a bi-Lipshitz homeomorphism, $T:X\rightarrow X$. We will see
that the sensitive dependence and positiveness of Lyapunov exponent in
order to argue the validity of these assumptions.

To that end, we assume that $T:X\to X$ satisfies the following
properties:
\begin{enumerate}
\item Sensitivity to initial conditions: There exists $\delta>0$ such
  that for $x \in X$, $\forall \epsilon >0$, there exists a $\sigma$-null
  (zero volume) set $\mathcal{V}(x)$ such that for all
  $ y\in B_\epsilon (x)\backslash \mathcal{V}$, there is
  $n(x,y)\in \bb{N}$ such that
  $d(T^{n(x,y)}x,T^{n(x,y)}y)>\delta$. And for $y\in \mathcal{V}(x)$,
  $d(T^{n}x,T^ny)\to 0$ as $n\to \infty$. (Note that this is a
  stronger notion than the one given in \cite{glasner1993sensitive}).
\item Positive Lyapunov exponent: If
  $y\in B^c_r (x)\backslash \mathcal{V}$ then $d(T^i x,T^iy)> \delta$
  for $i\sim \frac{1}{\lambda}\log\frac{\delta}{r}$, where $\lambda>0$
  plays the role of Lyapunov exponent (Note that this property is
  qualitative in nature).
\end{enumerate}
We give an informal argument using these properties to show that
Assumption~\ref{ad22} holds. Choose $r>0$ and fix $x$ and $y$ such that
$d(x,y)>r$. And also, define $a_n=d(T^n x,T^n y)$. We assume that
$\inf_n( a_n)=0$, otherwise \eqref{aexp} trivially holds for a given
$x,y$ and of $T$ . And also, we assume that
$\limsup_{n\to\infty} a_n>0$.

Let $\{n_k\}_{k\in\mathbb{N}}$ be a subsequence such that 
\begin{align*}
\inf_{k} a_{n_k}>0 \text{ and } \lim_{k\to\infty}a_{m_k}=0 \text{ with }  \{m_k\}_{k\in\mathbb{N}}:= \mathbb{N}\backslash \{n_k\}_{k\in\mathbb{N}}.
\end{align*}
From the assumption that $\limsup_{n\to\infty} a_n>0$ and $\inf_n a_n=0$, such a $\{n_k\}_{k\in\mathbb{N}}$ exists. Suppose that $n_{k+1}-n_{k}\to \infty$ as $k\to
\infty$, then by choosing
$k$ becomes large enough, cardinality of the set
$\left[n_k,n_{k+1} \right]\cap \{m_k\}_{k\in \bb{N}}$ can be made as
larger than any desired integer.

In other words, for every $\rho>0$, $1<M\in \bb{N}$, there exists $k_0$
such that for all $k\geq k_0$, we have
\begin{align}\label{finest}
  n_{k+1}-n_k >M^2\text{ and } a_m<\rho, \; \forall n_{k}< m<n_{k+1}.
\end{align}

Choosing $M= \frac{1}{\lambda}\log\frac{\rho}{r}$, $\bar{x}:=T^{n_k +1}x$ and $\bar{y}:=T^{n_k +1}y$, we see
that this violates the assumptions on the dynamical system. To see that, we firstly note that for $i\sim \frac{1}{\lambda}\log\frac{\rho}{r}=M$, we have
$d(T^i\bar{x},T^i \bar{y})>\rho$.  Now for  $n_{k}< m<n_{k+1}$, consider
\begin{align*}
a_{m}=d(T^{m-n_k-1}\bar{x},T^{m-n_k-1} \bar{y})=d(T^{m-n_k-1}T^{n_k+1}x,T^{m-n_k-1} T^{n_k+1} y) = d(T^{m}x,T^{m} y).
\end{align*}  
From Equation~\eqref{finest}, $a_m< \rho$ and  since $d(T^{i+n_k+1}{x},T^{i+n_k+1} {y})>\rho$ with $i\sim M$, we have a contradiction. Therefore, the supposition that
$n_{k+1}-n_{k}\to \infty$ as $k\to \infty$ is false and there exist a
positive constant, $J$ such that $n_{k+1}-n_{k}\leq J$ for any
$k$. This implies that cardinality of the set
$\{n_k \}_{k\in \bb{N}} \cap [1,2,3,...,N]$ is at least
$\lfloor\frac{N}{J}\rfloor$. As a result, we have the following
\begin{align*}
  D^2_N(x,y) &\geq \delta \sum_{\substack{k\in \bb{N}, \\ n_k<N}}\rho_{n_k\tau}
  \geq \delta\sum_{i=0}^{\lfloor\frac{N}{J}\rfloor}\rho_{i\tau}\geq \delta G\left(\alpha, J\right)\sum_{i=0}^N\rho_{i\tau},
\end{align*}
where $G\left(\alpha, J\right)>0$ depends only on $\alpha$ and
$J$. The above inequalities follow from non-decreasing property of
$\rho_t$, applying the lowest bound to any sum up to first
$\lfloor\frac{N}{J}\rfloor$ terms of an subsequence of a
non-decreasing sequence and the form of $\rho_t$.

To summarize, in the current section, we studied various filtering
models that satisfy the assumptions of Sections~\ref{S3} and
~\ref{S4}.

\section{Conclusions} \label{S6}

The problem that we studied in this paper is the asymptotic stability
of the nonlinear filter with deterministic dynamics. In order to
establish stability, we first proved, in Theorem~\ref{th1}, an
accuracy or consistency result for the smoother, \emph{i.e.}, the
convergence of the conditional distribution of the initial condition
given observations. We used this result to prove the stability of the
filter in Theorem~\ref{stab}. Using essentially identical methods, we
also established the accuracy of the smoother (Theorem~\ref{thd1}) and
the stability of the filter (Theorem~\ref{stabd}) in the case of
discrete time.

The main assumptions used in proving these results are quite natural
as discussed in Section~\ref{sig}, and are indeed satisfied by two
large classes of dynamical systems, as discussed in
Section~\ref{S7}. In particular, these assumptions are valid for a class of
diffeomorphisms of compact manifolds with appropriate enough
observation function, as well as a  class of nonlinear
differential equations that includes models such as the Lorenz models
(using numerical evidence for Assumption~\ref{a22}).

There are various possible directions for further
studies. Theorems~\ref{stab} and~\ref{stabd} do not give any rate
of convergence, because of the use of Martingale convergence theorem,
and it would be interesting to find finer methods that may give the
rate of convergence, such as those \cite[Section
4.3]{bocquet2017degenerate} available for the convergence of
covariance of the filter for linear models. Further, partly because of
the use of convergence of the smoother to prove filter stability, our
results do not give much information about the structure of the
asymptotic filtering distribution, such as that which is available
\cite[Sections 4.3, 5]{bocquet2017degenerate}, \cite[Remark
3.2]{reddy2019asymptotic} for the linear filter. We hope that further
investigations in this direction will lead to efficient numerical
implementations of the filter for deterministic dynamics, especially
for high dimensional systems.

\section*{Acknowledgements}
The authors would like to thank Amarjit Budhiraja for valuable
discussions and pointing to the work of C{\'e}rou
\cite{cerou2000long}. The authors would also like to thank Chris Jones
and Erik Van Vleck for inputs, and The Statistical and Applied
Mathematical Sciences Institute (SAMSI), Durham, NC, USA where a part
of the work was completed. ASR’s visit to SAMSI was supported by
Infosys Foundation Excellence Program of ICTS. AA acknowledges support from US Office of Naval Research under grant N00014-18-1-2204. Authors acknowledge the
support of the Department of Atomic Energy, Government of India, under
projects no.12-R\&D-TFR-5.10-1100, and no.RTI4001. The authors also thank the anonymous referees for their thoughtful comments that helped improve the manuscript.
 
\bibliographystyle{plainnat}
\bibliography{20210822-nonlin-filter-stability-aa}

\end{document}